\documentclass{amsart}
\usepackage{amsmath}
\usepackage{paralist}
\usepackage{graphics} 

\usepackage{graphicx}  
\usepackage[colorlinks=true]{hyperref}
   \usepackage{amsthm}
\usepackage{amsmath}%
\usepackage{amsfonts}%
\usepackage{amssymb}%
\usepackage{amsaddr}
\usepackage{graphicx}
\usepackage{mathabx}
\usepackage{enumitem}
\usepackage{bbm}
\usepackage{tikz}
\usepackage{epsfig}
\usepackage{subcaption}
\usepackage{makecell}
\usepackage{setspace}

\hypersetup{urlcolor=blue, citecolor=red}

\newtheorem{theorem}{Theorem}
\theoremstyle{plain}

\newtheorem{corollary}{Corollary}

\newtheorem{definition}{Definition}

\newtheorem{lemma}{Lemma}

\newtheorem{proposition}{Proposition}
\newtheorem{remark}{Remark}

\numberwithin{equation}{section}

\newcommand{\ddfrac}[2]{\frac{\mathrm{d}#1}{\mathrm{d}#2}}

\author[Y. Holle, M. Herty and M. Westdickenberg]{}

 \email{holle@eddy.rwth-aachen.de}
 \email{herty@igpm.rwth-aachen.de}
 \email{mwest@instmath.rwth-aachen.de}

\thanks{This work has been funded by the Deutsche Forschungsgemeinschaft (DFG, German Research Foundation) projects 320021702/GRK2326 Energy, Entropy, and Dissipative Dynamics (EDDy) and HE5386/18,19.}

\thanks{$^*$ Corresponding author: Yannick Holle}
\title{New Coupling Conditions for Isentropic Flow on Networks}


\begin{document}
\maketitle
\centerline{\scshape Yannick Holle$^*$}
\medskip
{\footnotesize
 \centerline{Institut f\"ur Mathematik}
   \centerline{RWTH Aachen University}
   \centerline{Templergraben 55, 52062 Aachen, Germany}
} 

\medskip

\centerline{\scshape Michael Herty}
\medskip
{\footnotesize
 \centerline{Institut f\"ur Geometrie und Praktische Mathematik}
   \centerline{RWTH Aachen University}
   \centerline{Templergraben 55, 52062 Aachen, Germany}
} 

\medskip 

\centerline{\scshape Michael Westdickenberg}
\medskip
{\footnotesize
 \centerline{Institut f\"ur Mathematik}
   \centerline{RWTH Aachen University}
   \centerline{Templergraben 55, 52062 Aachen, Germany}
} 

\begin{abstract}
We introduce new coupling conditions for isentropic flow on networks based on an artificial density at the junction. The new coupling conditions can be derived from a kinetic model by imposing a condition on energy dissipation. Existence and uniqueness of solutions to the generalized Riemann and Cauchy problem are proven. The result for the generalized Riemann problem is globally in state space. Furthermore, non-increasing energy at the junction and a maximum principle are proven. A numerical example is given in which the new conditions are the only known conditions leading to the physically correct wave types. The approach generalizes to full gas dynamics.
\end{abstract}

\section{Introduction}
\label{sec:Introduction}
We consider networks modeled by a directed graph where the dynamics on each edge are described by one-dimensional conservation laws. The dynamics are coupled at the vertices of the graph, called junctions. We are especially interested in (isentropic) gas dynamics, but there are many
other applications for example in traffic, supply chains, data networks or blood circulation. This field became
of interest to many researchers in the last two decades, see for example the overview by Bressan et al. \cite{BCGHP2014}. A main challenge is posed by prescribing suitable coupling conditions at the junction. We consider novel conditions for the system of isentropic gas, also referred to as p-system. \medskip \\
The isentropic gas equations at a junction with $k=1,\dots,d\in\mathbb N$ adjacent pipelines are given by
\begin{equation}\label{eq:IsentropicGasinIntroduction}
\begin{cases}
	\partial_t \rho_k+\partial_x(\rho_k u_k)&=0\\
	\partial_t (\rho_k u_k)+\partial_x(\rho_k u_k^2+\kappa \rho_k^\gamma)&=0
\end{cases}\quad\text{for a.e. } t>0, x>0,
\end{equation}
where $\rho_k\ge0$ denotes the gas density, $u_k\in\mathbb R$ the mean velocity, and $p=\kappa \rho^\gamma$ the pressure given by the $\gamma$-pressure law with $\kappa>0$ and $1<\gamma<3$. 
The equation is supplemented by an entropy condition, an initial condition $(\hat\rho_k,\hat u_k)(x)=(\rho_k,u_k)(0+,x)$, $x>0$ and a suitable coupling condition on the traces $(\bar\rho_k,\bar u_k)(t)=(\rho_k,u_k)(t,0+)$, $t>0$. Furthermore, we are interested in weak solutions to (\ref{eq:IsentropicGasinIntroduction}). For a general introduction to the theory of conservation laws see the book by Dafermos \cite{Da2010}.

\subsection{Previous results}
The most challenging problem in modeling (gas) networks is to find physically correct coupling conditions. A first condition is usually conservation of mass at the junction
\begin{equation}\label{eq:ConservationofMassIntroduction}
\sum_{k=1}^d A_k \bar\rho_k\bar u_k=0,\quad\text{for a.e. }t>0,
\end{equation}
where $A_k>0$ denotes the cross-sectional area of the $k$-th pipeline. To ensure uniqueness of solutions, we impose more conditions at the junction. The number of additional conditions depends on the sign of the characteristic speeds at the junction. Several coupling conditions have been proposed, for example equality of pressure \cite{BHK2006II,BHK2006}
\begin{equation}\label{eq:defEqualPressure}
\kappa \bar\rho_k^\gamma=\mathcal H_p(t),\quad  \text{for }k=1,\dots, d,\quad\text{a.e. }t>0,
\end{equation}
equality of momentum flux \cite{CoGa2006, CoGa2008}
\begin{equation}
\bar\rho_k\bar u_k^2+\kappa \bar\rho_k^\gamma=\mathcal H_{\mathrm{MF}}(t),\quad \text{for }k=1,\dots, d,\quad\text{a.e. }t>0,
\end{equation}
and equality of stagnation enthalpy/Bernoulli invariant \cite{RFHY2015}
\begin{equation}\label{eq:defEqualStagnationEnthalpy}
\frac{\bar u_k^2}{2}+\frac{\kappa}{\gamma-1} \bar\rho_k^{\gamma-1}=\mathcal H_{\mathrm{SE}}(t),\quad  \text{for }k=1,\dots, d,\quad\text{a.e. }t>0.
\end{equation}
Notice that the first two conditions are non-physically in the sense that energy may be produced at the junction \cite{RFHY2015}. Equality of stagnation enthalpy implies conservation of energy at the junction. We derive a coupling condition with dissipated energy at the junction. This is consistent with the isentropic gas equations where energy is also dissipated. The special situation with only two pipelines were studied in \cite{HoRi1999,CoHo2016}\medskip\\
Usually existence and uniqueness of solutions to a generalization of the Riemann problem at the junction are studied locally in state space first. Reigstad \cite{Re2015} introduced a method to study existence and uniqueness almost globally in the subsonic region under a technical assumption. The results for the Riemann problem are used, to construct approximate solutions to the generalized Cauchy problem, usually by wave front tracking. See the book by Bressan \cite{Br2000} for a general introduction to the wave front tracking method. Colombo, Herty and Sachers \cite{CHS2008} proved a general existence and uniqueness theorem for the generalized Cauchy problem by using this method. This theorem requires a transversality condition, subsonic data and sufficiently small total variation of the initial data. \\
Another approach to prove existence of solutions is the method of compensated compactness. This has been applied to a scalar traffic model \cite{CoGa2010} and the isentropic gas equations \cite{Ho2020}. The method requires less assumptions on the regularity of the initial data, but is restricted to systems with a large class of entropies. Moreover, less regularity of the solutions is obtained and the traces at the junction have to be considered carefully, see e.g. \cite{BeBo2002Boundary,Ho2020}.

\subsection{A new coupling condition}
To supplement conservation of mass (\ref{eq:ConservationofMassIntroduction}), we use an approach based on the kinetic model for isentropic gas and a maximum energy/entropy dissipation principle at the junction. \medskip\\
A kinetic model for the isentropic gas equations was introduced by Lions, Perthame and Tadmor \cite{LPT1994}. The corresponding vector-valued BGK model were introduced by Bouchut \cite{Bo1999}. For $f=f(t,x,\xi)$ we impose
\begin{equation}\label{eq:KineticModel}
\partial_t f^k_\epsilon+\xi\partial_x f^k_\epsilon=\frac{M[f^k_\epsilon]-f^k_\epsilon}{\epsilon},\quad \text{for a.e. }t>0,x>0,\xi\in\mathbb R,k=1,\dots,d,
\end{equation}
with Maxwellian $M[f]$ (will be defined later). The half-space solutions are coupled at the junction by a kinetic coupling condition
\begin{equation}
\Psi^k[f_\epsilon^1(t,0,\cdot),\dots,f_\epsilon^d(t,0,\cdot)](\xi)=f_\epsilon^k(t,0,\xi),\quad \text{for a.e. }t>0,\xi>0.
\end{equation}
To select the function $\Psi$, we follow an idea of Dafermos \cite{Da1973} and use maximum entropy/energy dissipation as a selection criteria for the physically correct kinetic coupling condition. More precisely, we determine $\Psi^k$ such that as much energy is dissipated as possible under the condition of conserved mass. We obtain that for every pipeline the outgoing data is given by a Maxwellian with an artificial density $\rho_*^\epsilon$ and zero speed, i.e.
\begin{equation}
\Psi^k[f_\epsilon^1(t,0,\cdot),\dots,f_\epsilon^d(t,0,\cdot)](\xi)=M(\rho_*^\epsilon(t),0,\xi),\quad \text{for a.e. }t>0,\xi>0.
\end{equation}
A formal limit argument leads to the definition of a generalized Riemann problem for $(\rho_k,u_k)$. As in \cite{HeRa2006}, each half-space solution is given by the restriction of the solution to a standard Riemann problem. The left Riemann initial state is given again by an artificial state with a suitable density and zero speed.
\begin{definition}\label{def:GeneralizedRiemann}
Let $(\hat\rho_k,\hat u_k)\in D$, $k=1,\dots,d$. Then, we call $(\rho_k,u_k)\colon (0,\infty)_t\times(0,\infty)_x\to D$ a weak solution to the generalized Riemann problem if the following assertions hold true:
\begin{itemize}
		\item [RP0:] The solution satisfies the constant initial condition
\begin{equation*}
(\rho_k,u_k)(0+,x)=(\hat{\rho}_k,\hat{u}_k)\in D,\quad \text{for all }x>0,\,k=1,\dots,d.
\end{equation*}
    \item [RP1:] There exists $\rho_*\ge 0$ such that $(\rho_k,u_k)$ is equal to the restriction to $x>0$ of the weak entropy solution in the sense of Lax with initial condition
			\begin{equation*}
				(\rho_k,u_k)(0+,x)=
				\begin{cases}
					(\hat{\rho}_k,\hat{u}_k),\quad& x>0,\\
					(\rho_*,0),\quad& x<0,
				\end{cases}
			\end{equation*}
			for all $k=1,\dots,d$.
    \item [RP2:] Mass is conserved at the junction
			\begin{equation*}
				\sum_{k=1}^d A_k\bar{\rho}_k\bar{u}_k=0,
			\end{equation*}
			where $(\bar{\rho}_k,\bar{u}_k)=(\rho_k,u_k)(t,0+)\in D,$ for a.e. $t>0$, $k=1,\dots,d$.
\end{itemize}
\end{definition}

The set $D$ is the state space and will be defined in Section \ref{sec:IsentropicGas}. Notice that the condition \textit{RP1} can be reformulated using a Riemann problem formulation for boundary conditions $\mathcal V(\rho_*,0)$. This formulation was used by Dubois and LeFloch \cite{DuLe1988} and will be defined later. It is illustrated in Figure \ref{fig:BoundaryRiemann}. The set $\mathcal V(\rho_*,0)$ will be used to define solutions to the generalized Cauchy problem (see Definition \ref{def:GeneralizedCauchyProblem}).\medskip\\
Since the new coupling condition is based on restrictions of standard Riemann problems, we get a simple wave structure for the solutions in the sense of Definition \ref{def:GeneralizedRiemann}. This structure allows us to prove existence and uniqueness of solutions globally in state space. We can use techniques by Reigstad \cite{Re2015} in the subsonic regime and extend them to the full state space. A general local existence and uniqueness result \cite{CHS2008} for the Cauchy problem applies to the new condition. As a by-product we obtain Lipschitz continuous dependence on the initial data.\medskip\\
The coupling condition satisfies several properties regarding the energy/entropy dissipation. First, we obtain that entropy is non-increasing at the junction for a large class of symmetric entropies and in particular for the physical energy. A corollary of this property is a maximum principle on the Riemann invariants. More precisely, if the Riemann invariants of the initial data are bounded, then the Riemann invariants of the solution are bounded for all times. Furthermore, we prove a relation between the traces of the stagnation enthalpy at the junction.\medskip\\
We present an example in which the new coupling condition is the only condition leading to the physically correct wave types. The solutions to the generalized Riemann problem are computed numerically by using Newton's method. Furthermore, we study level sets associated to different coupling conditions and consider their geometry.\medskip\\
Our approach can be easily generalized to other hyperbolic systems. We extend it to full gas dynamics and obtain a similar coupling condition with an artificial density, zero speed and an artificial temperature at the junction. For more details and a brief literature overview see Section \ref{sec:FullGasDynamics}.

\subsection{Organization of the paper}
In Section \ref{sec:IsentropicGas}, we recall several properties of the isentropic gas equations and the initial boundary value problem. In Section \ref{sec:MotivationandDerivation}, we give a detailed motivation and a formal derivation of the coupling condition. Existence and uniqueness of solutions to the generalized Riemann problem will be proven in Section \ref{sec:RiemannProblem} and the corresponding results for the generalized Cauchy problem will be proven in Section \ref{sec:CauchyProblem}. In Section \ref{sec:PropertiesCoupling}, we derive several physical properties of the coupling condition, e.g. non-increasing energy, a maximum principle on the Riemann invariants and a relation for the traces of the stagnation enthalpy. Section \ref{sec:NumericalResults} is devoted to numerical considerations. In Section \ref{sec:FullGasDynamics}, the extension of our approach to full gas dynamics is given. In Section \ref{sec:OutlookRemarks}, we finish with a conclusion.

\section{The isentropic gas equations and basic definitions}
\label{sec:IsentropicGas}
\subsection{Entropy solutions, Riemann invariants and Lax curves}
\label{subsec:EntropyRiemannLax}
The isentropic gas equations in one space dimension are given by
\begin{equation}
\begin{cases}
	\partial_t \rho+\partial_x(\rho u)&=0\\
	\partial_t (\rho u)+\partial_x(\rho u^2+\kappa \rho^\gamma)&=0
\end{cases}\quad \text{a.e. }t\ge 0, x\in \mathbb R.
\end{equation}
Furthermore, we impose the entropy condition 
\begin{equation}
\partial_t\eta_S(\rho,u)+\partial_x G_S(\rho,u)\le0\quad\text{a.e. }t, x,
\end{equation}
for all (weak) entropy pairs $(\eta_S,G_S)$, where $\eta_S$ is a convex function with
\begin{equation}
G_S'=\eta_S' F',\quad \eta_S(\rho=0,u)=0,\quad \eta_{S,\rho}'(\rho=0,u)=S(u),\quad\text{ for all }u\in\mathbb R,
\end{equation}
and a suitable convex function $S\colon\mathbb R\to\mathbb R$. The involved derivatives are taken with respect to the conserved quantities $(\rho,\rho u)$. We recall some basic definitions and notation:
\begin{gather}
D=\{(x_0, x_1)\in \mathbb R^2|x_0>0 \text{ or }x_0=x_1=0\},\\
\chi(\rho,\xi)=c_{\gamma,\kappa}(a_\gamma^2 \rho^{\gamma-1}-\xi^2)_+^\lambda,\\
\theta=\frac{\gamma-1}{2},\quad\lambda=\frac{1}{\gamma-1}-\frac{1}{2},\quad c_{\gamma,\kappa}=\frac{a_\gamma^{2/(\gamma-1)}}{J_\lambda},\\
J_\lambda=\int_{-1}^1 (1-z^2)^\lambda\,\mathrm{d}z,\quad a_\gamma=\frac{2\sqrt{\gamma \kappa}}{\gamma-1}.
\end{gather}
The isentropic gas equations admit the Riemann invariants 
\begin{align}
\omega_1&=u-a_\gamma\rho^\theta,\quad\omega_2=u+a_\gamma\rho^\theta,
\end{align}
for $(\rho,u)\in D$. The eigenvalues are given by
\begin{align}
\lambda_1(\rho,u)=u-\sqrt{\kappa\gamma}\rho^\theta,\quad\lambda_2(\rho,u)=u+\sqrt{\kappa\gamma}\rho^\theta,
\end{align}
and the eigenvectors by
\begin{align}
r_1(\rho,u)=\begin{pmatrix}1 \\u-\sqrt{\kappa\gamma}\rho^\theta \end{pmatrix},\quad r_2(\rho,u)=\begin{pmatrix}1 \\u+\sqrt{\kappa\gamma}\rho^\theta \end{pmatrix},
\end{align}
for $(\rho,u)\in D$. We call a state $(\rho,u)\in D$
\begin{itemize}
\item \textit{subsonic} if $\lambda_1(\rho,u)<0<\lambda_2(\rho,u)$;
\item \textit{sonic} if $\lambda_1(\rho,u)=0$ or $\lambda_2(\rho,u)=0$;
\item \textit{supersonic} if $0<\lambda_1(\rho,u)<\lambda_2(\rho,u)$ or $\lambda_1(\rho,u)<\lambda_2(\rho,u)<0$.
\end{itemize}
Next, we define several quantities corresponding to the kinetic (BGK) model for isentropic gas dynamics (\ref{eq:KineticModel}).
The vector-valued Maxwellian $M[f]$ for $f\colon \mathbb R\to D$ is defined by
\begin{equation}
M[f](\xi)=M(\rho_f,u_f,\xi)
\end{equation} 
where
\begin{equation}
\rho_f=\int_{\mathbb R}f_0(\xi)\,\mathrm{d}\xi,\quad\rho_f u_f=\int_{\mathbb R}f_1(\xi)\,\mathrm{d}\xi
\end{equation}
and 
\begin{gather}
M(\rho,u,\xi)=(\chi(\rho,\xi-u),((1-\theta)u+\theta \xi)\chi(\rho,\xi-u)).
\end{gather}
The kinetic entropies are defined by 
\begin{align}
H_S(f,\xi)&=\int_\mathbb R \Phi(\rho(f,\xi),u(f,\xi),\xi,v)S(v)\;\text{d}v\quad \text{for }f\in D\backslash \{ 0\},\\
H_S(0,\xi)&=0,
\end{align}
where
\begin{align}
u(f,\xi)&=\frac{f_1/f_0-\theta \xi}{1-\theta},\\
\rho(f,\xi)&=a_\gamma^{-\tfrac{2}{\gamma-1}}\left(\left(\frac{f_1/f_0-\xi}{1-\theta}\right)^2+\left(\frac{f_0}{c_{\gamma,\kappa}}\right)^{1/\lambda}\right)^{\tfrac{1}{\gamma-1}}.
\end{align}
The kernel $\Phi$ is defined by
\begin{gather}
\Phi(\rho,u,\xi,v)=\frac{(1-\theta)^2}{\theta}\frac{c_{\gamma,\kappa}}{J_\lambda}\mathbbm{1}_{\omega_1<\xi<\omega_2}\mathbbm{1}_{\omega_1<v<\omega_2}|\xi-v|^{2\lambda-1}\Upsilon_{\lambda-1}(z),\\
z=\frac{(\xi+v)(\omega_1+\omega_2)-2(\omega_1\omega_2+\xi v)}{(\omega_2-\omega_1)|\xi-v|},\\
\Upsilon_{\lambda-1}(z)=\int_1^z (y^2-1)^{\lambda -1}\;\text{d}y,\quad z\ge 1.
\end{gather}
The kinetic Riemann invariants are given by
\begin{align}
\omega_1&=u(f,\xi)-a_\gamma\rho(f,\xi)^\theta,\quad\omega_2=u(f,\xi)+a_\gamma\rho(f,\xi)^\theta,
\end{align}
for $f\in D\backslash \{ 0\}$. The macroscopic entropy and entropy flux are given by
\begin{align}
\eta_S(\rho,u)&=\int_{\mathbb R}\chi(\rho,v-u) S(v)\;\text{d}v=\int_{\mathbb R} H_S(M(\rho,u,\xi),\xi)\;\text{d}\xi,\\
G_S(\rho,u)&=\int_\mathbb R [(1-\theta)u+\theta v]\ \chi(\rho,v-u)S(v)\;\text{d}v\\
&=\int_\mathbb R\xi H_S(M(\rho,u,\xi),\xi)\;\text{d}\xi,
\end{align}
for $(\rho,u)\in D$. If additionally $S\in C^1(\mathbb R,\mathbb R)$, then the gradient of $\eta$ with respect to the conserved variables is given by
\begin{equation}
\eta_S'(\rho,u)=\frac{1}{J_\lambda}\int_{-1}^1 (1-z^2)^\lambda \begin{pmatrix} S(u+a_\gamma\rho^\theta z)+(\theta a_\gamma\rho^\theta z- u)S'(u+a_\gamma\rho^\theta z)\\ S'(u+a_\gamma\rho^\theta z)\end{pmatrix}\mathrm{d}z,
\end{equation}
for $(\rho,u)\in D$. The kinetic entropy parametrized by $S(v)=v^2/2$ is given by
\begin{equation}\label{eq:defkineticenergy}
H(f,\xi)=\frac{\theta}{1-\theta}\frac{\xi^2}{2} f_0+\frac{\theta}{2c_{\gamma,\kappa}^{1/\lambda}}\frac{f_0^{1+1/\lambda}}{1+1/\lambda}
+\frac{1}{1-\theta}\frac{1}{2}\frac{f_1^2}{f_0}-\frac{\theta}{1-\theta}\xi f_1,
\end{equation}
and the corresponding macroscopic entropy pair is given by the physical energy and energy flux
\begin{align}
\eta(\rho,u)=\frac{\rho u^2}{2}+\frac{\kappa}{\gamma-1}\rho^\gamma,\quad G(\rho,u)=\frac{\rho u^3}{2}+\frac{\gamma\kappa}{\gamma-1}\rho^\gamma u.
\end{align}
To construct solutions to the generalized Riemann problem, we need the (forward) Lax wave curves $\mathcal W_1(\rho_0,u_0)$ and $\mathcal W_2(\rho_0,u_0)$ which are the composition of the corresponding rarefaction and shock curves. The rarefaction curves are given by
\begin{align}
\mathcal R_1(\rho_0,u_0):\quad u=u_0+a_\gamma\rho_0^\theta-a_\gamma\rho^\theta,\quad \text{for }\rho<\rho_0,\\
\mathcal R_2(\rho_0,u_0):\quad u=u_0-a_\gamma\rho_0^\theta+a_\gamma\rho^\theta,\quad \text{for }\rho>\rho_0,
\end{align}
and the shock curves are given by
\begin{align}
\mathcal S_1(\rho_0,u_0):\quad u=u_0-\sqrt{\frac{\kappa(\rho^\gamma-\rho_0^\gamma)(\rho-\rho_0)}{\rho\rho_0}},\quad \text{for }\rho>\rho_0,\\
\mathcal S_2(\rho_0,u_0):\quad u=u_0-\sqrt{\frac{\kappa(\rho^\gamma-\rho_0^\gamma)(\rho-\rho_0)}{\rho\rho_0}},\quad \text{for }\rho<\rho_0.
\end{align}
We will use the notation $\mathcal S_2^-(\rho,u), \mathcal R_2^-(\rho,u),\dots$ for the reversed wave curves. They satisfy the same condition as the forward wave curves but the fixed variables are $(\rho,u)$ instead of $(\rho_0,u_0)$.
We always consider the self-similar solutions to Riemann problems in the sense of Lax and denote them by $\mathcal{RP}(\rho_l,u_l,\rho_r,u_r)(t/x)$ for initial data
\begin{equation}
\begin{cases}
(\rho_l,u_l),\quad &x<0,\\
(\rho_r,u_r),\quad &x>0,
\end{cases}
\end{equation} 
where $(\rho_l,u_l),(\rho_r,u_r)\in D$.

\subsection{The initial boundary value problem}
In this subsection, we recall some basic properties of the initial boundary value problem.
The sets $\mathcal E(\rho_b,u_b)$ and $\mathcal V(\rho_b,u_b)$ of admissible boundary values were introduced in \cite{DuLe1988}. We recall their definitions.
\begin{definition}
Let $(\rho_b,u_b)\in D$. $\mathcal V(\rho_b,u_b)$ is the set of states $(\rho,u)\in D$ with
\begin{equation*}
(\rho,u)=\mathcal{RP}(\rho_b,u_b,\rho_r,u_r)(0+),\quad\text{for a state }(\rho_r,u_r)\in D.
\end{equation*}
\end{definition}
\begin{definition}
Let $(\rho_b,u_b)\in D$. $\mathcal E(\rho_b,u_b)$ is the set of states $(\rho,u)\in D$ with
\begin{equation*}
G_S(\rho,u)-G_S(\rho_b,u_b)-\eta_S'(\rho_b,u_b)\left(F(\rho,u)-F(\rho_b,u_b)\right)\le 0
\end{equation*}
for all entropy pairs $(\eta_S,G_S)$ of class $C^1$ (i.e. $S\in C^1(\mathbb R,\mathbb R$)).
\end{definition}
We recall the following result.
\begin{proposition}[{\cite[Theorem 3.4]{KSX1997}}]
Let $(\rho_b,u_b)\in D$, then $\mathcal V(\rho_b,u_b)\subset \mathcal E(\rho_b,u_b)$. The reversed set inclusion does not hold true in general.
\end{proposition}
Next, we define subsets of $D$ to consider different situations in the initial boundary value problem. We are especially interested in the case $(\rho_b,u_b)=(\rho_*,0)$.
\begin{definition}\label{def:SetsBoundaryValueProb}
Let $(\rho_b,u_b)\in D$. Then,
\begin{itemize}
\item $\mathcal A$ is the set of states which are connected to $\mathcal W_1(\rho_b,u_b)\cap\{\lambda_1\ge 0\}$ by its reversed 2-wave curve;
\item $\mathcal B$ is the set of states which are connected to $\mathcal W_1(\rho_b,u_b)\cap\{\lambda_1<0<\lambda_2\}$ by its reversed 2-wave curve with positive wave speed;
\item $\mathcal C$ is the set of states which are connected to $\mathcal W_1(\rho_b,u_b)\cap\{\lambda_2\le 0\}$ by its reversed 2-wave curve or are connected to $\mathcal W_1(\rho_b,u_b)\cap\{\lambda_1<0<\lambda_2\}$ by its reversed 2-wave curve with non-positive wave speed;
\item $\mathcal J$ is the set of states which are connected to $\mathcal W_1(\rho_b,u_b)\cap\{\lambda_1<0<\lambda_2\}$ by its reversed 2-wave curve with zero wave speed.
\end{itemize}
We write  $\mathcal A_*,\, \mathcal B_*,\, \mathcal C_*,\,\mathcal J_*$ if $(\rho_b,u_b)=(\tilde\rho_*,0)$ with $\tilde\rho_*\ge 0$. 
\end{definition}

The sets $\mathcal A_*,\, \mathcal B_*,\, \mathcal C_*,\,\mathcal J_*$ are shown in Figure \ref{fig:ABCJ}. The sets $\mathcal A_*$ and $\mathcal B_*$ are separated by the 2-wave curve $\mathcal W_2(\rho_\alpha,u_\alpha)$, where $(\rho_\alpha,u_\alpha)$ is the unique state in $\{\lambda_1=0\}\cap \mathcal R_1(\tilde\rho_*,0)$. The sets $\mathcal B_*$ and $\mathcal C_*$ are separated by $\mathcal J_*$ and $\mathcal R_2(\rho_\beta,u_\beta)$, where $(\rho_\beta,u_\beta)$ is the unique state in $\{\lambda_2=0\}\cap \mathcal S_1(\tilde\rho_*,0)$. For the construction of $(\rho_\alpha,u_\alpha)$ and $(\rho_\beta,u_\beta)$, we refer to Figure \ref{fig:rhou12Construction}.

Next, we construct a solution which satisfies all properties in Definition \ref{def:GeneralizedRiemann} except of conservation of mass \textit{RP2}. They will be used to construct the desired solution to Definition \ref{def:GeneralizedRiemann} later. To clarify that \textit{RP2} does not necessarily hold true, we denote the artificial density by $\tilde\rho_*$.

\begin{figure}
\centering
\includegraphics[trim=1cm 2cm 1cm 2cm,clip=true,width=0.8\textwidth]{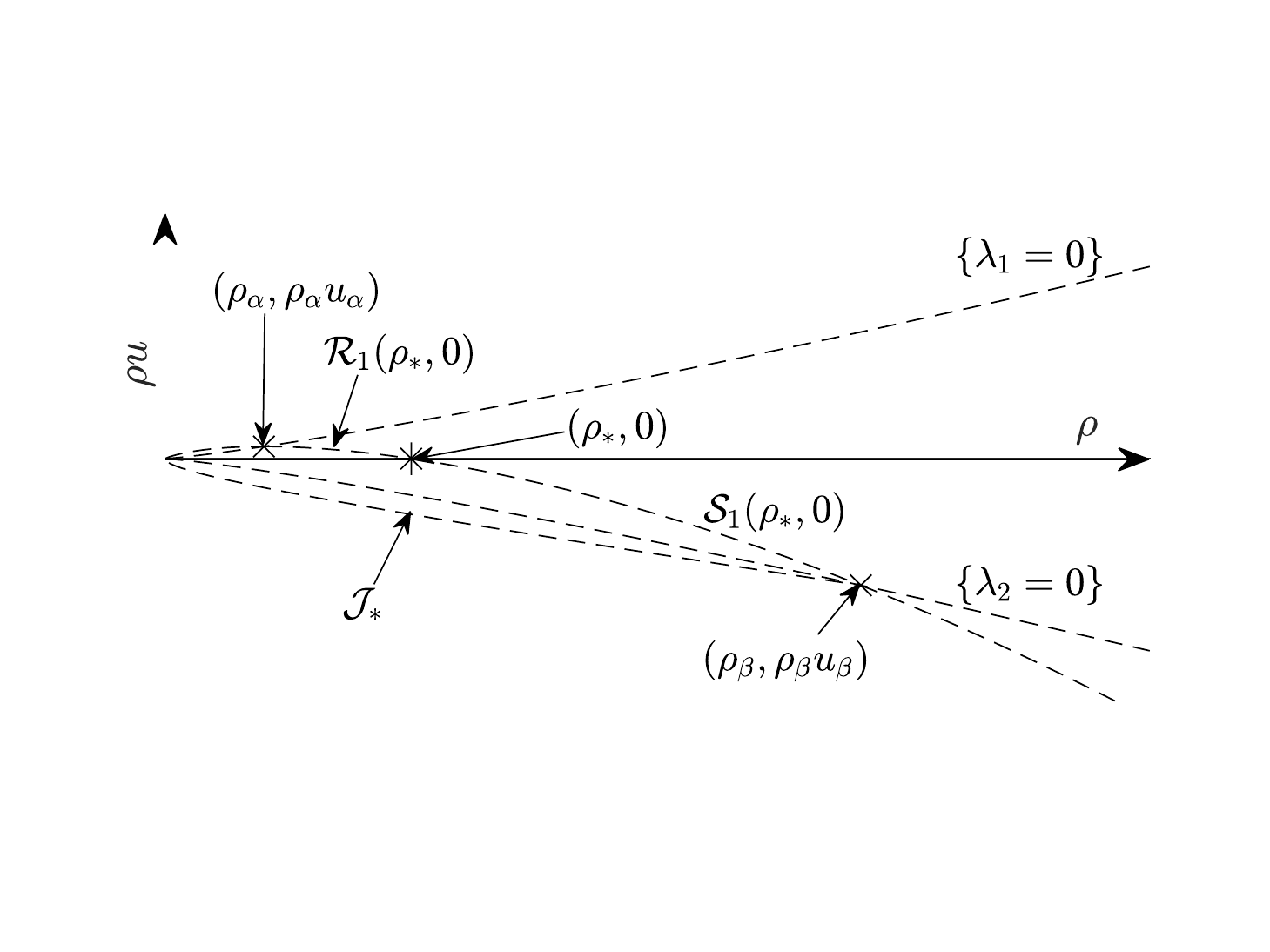}
\caption{Construction of $(\rho_\alpha,u_\alpha)$ and $(\rho_\beta, u_\beta)$}
\label{fig:rhou12Construction}
\end{figure}

\begin{figure}
\centering
\includegraphics[trim=1cm 2cm 1cm 2cm,clip=true,width=0.8\textwidth]{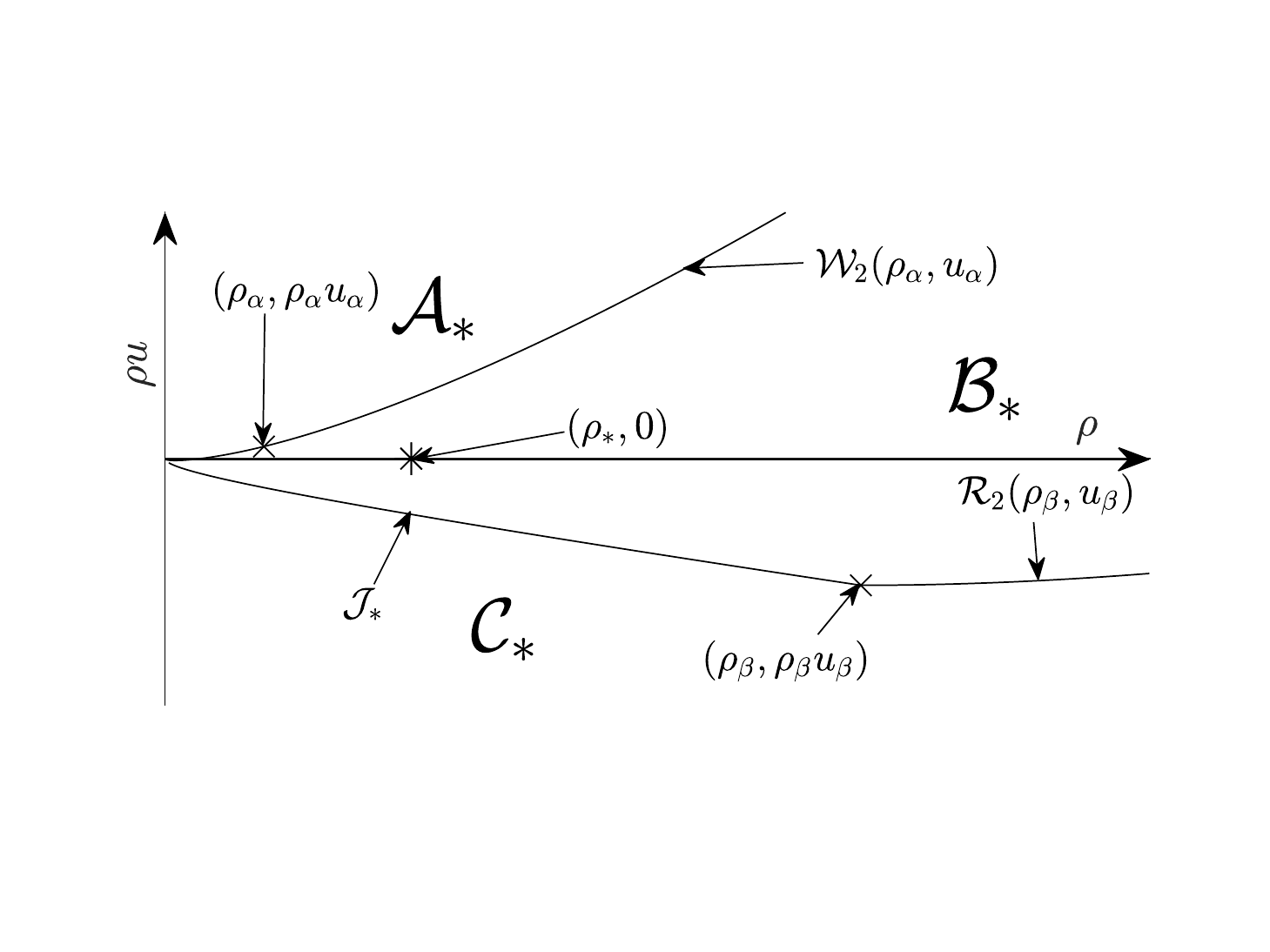}
\caption{The sets $\mathcal A_*$, $\mathcal B_*$, $\mathcal C_*$ and $\mathcal J_*$}
\label{fig:ABCJ}
\end{figure}

\begin{lemma}\label{lemma:ExistenceFixedRho}
Let $(\hat \rho_k,\hat u_k)\in D$ and $\tilde\rho_*\ge 0$. Then, there exists a unique function $(\rho_k,u_k)\colon(0,\infty)_t\times(0,\infty)_x\to D$ which coincides with the self-similar Lax solution to the standard Riemann problem with initial condition 
\begin{equation*}
(\rho_k,u_k)(0,x)=\begin{cases}
(\hat\rho_k,\hat u_k),\quad &x>0,\\
(\tilde\rho_*,0),\quad &x<0,
\end{cases}
\end{equation*}
for a.e. $t>0$, $x>0$.
Furthermore, we have the following properties for the trace $(\bar\rho_k,\bar u_k)=(\rho_k,u_k)(t,0+)$ illustrated in Figure \ref{fig:BoundaryRiemann}:
\begin{enumerate}[label=(\roman*)]
\item $(\hat\rho_k,\hat u_k)= (\bar\rho_k,\bar u_k)$ if and only if $(\hat\rho_k,\hat u_k)\in \mathcal{V}(\tilde\rho_*,0)$.
\item $(\bar\rho_k,\bar u_k)$ cannot be supersonic with $\lambda_1(\bar\rho_k,\bar u_k)> 0$.
\item $(\bar\rho_k,\bar u_k)$ is sonic with $\lambda_1(\bar\rho_k,\bar u_k)=0$ if and only if $(\hat\rho_k,\hat u_k)\in\mathcal A$. Furthermore, $(\bar\rho_k,\bar u_k)$ is the unique element in $\{\lambda_1=0\}\cap \mathcal R_1(\tilde\rho_*,0)$.
\item $(\bar\rho_k,\bar u_k)$ is subsonic if and only if $(\hat\rho_k,\hat u_k)\in\mathcal B$.
\item $(\bar\rho_k,\bar u_k)$ is sonic with $\lambda_2(\bar\rho_k,\bar u_k)=0$ if and only if $(\bar\rho_k,\bar u_k)$ is connected to $(\hat\rho_k,\hat u_k)$ by a 2-rarefaction curve and $(\hat\rho_k,\hat u_k)\in\mathcal C\backslash \{\lambda_2\ge 0\}$.
\item $(\bar\rho_k,\bar u_k)$ is supersonic with $\lambda_2(\bar\rho_k,\bar u_k)< 0$ if and only if $(\hat\rho_k,\hat u_k)\in\mathcal C\cap\{\lambda_2<0\}$.
\end{enumerate}
\end{lemma}
\begin{proof}
The existence and uniqueness of self-similar Lax solutions to Riemann problems is well-known. For Riemann problems with vacuum initial data see \cite{LiSm1980}. It remains to prove the properties $(\bar\rho_k,\bar u_k)$. They follow from the considerations in \cite{DuLe1988}.
\end{proof}

\begin{figure}
\centering
\includegraphics[trim=1cm 2cm 1cm 2cm,clip=true,width=0.8\textwidth]{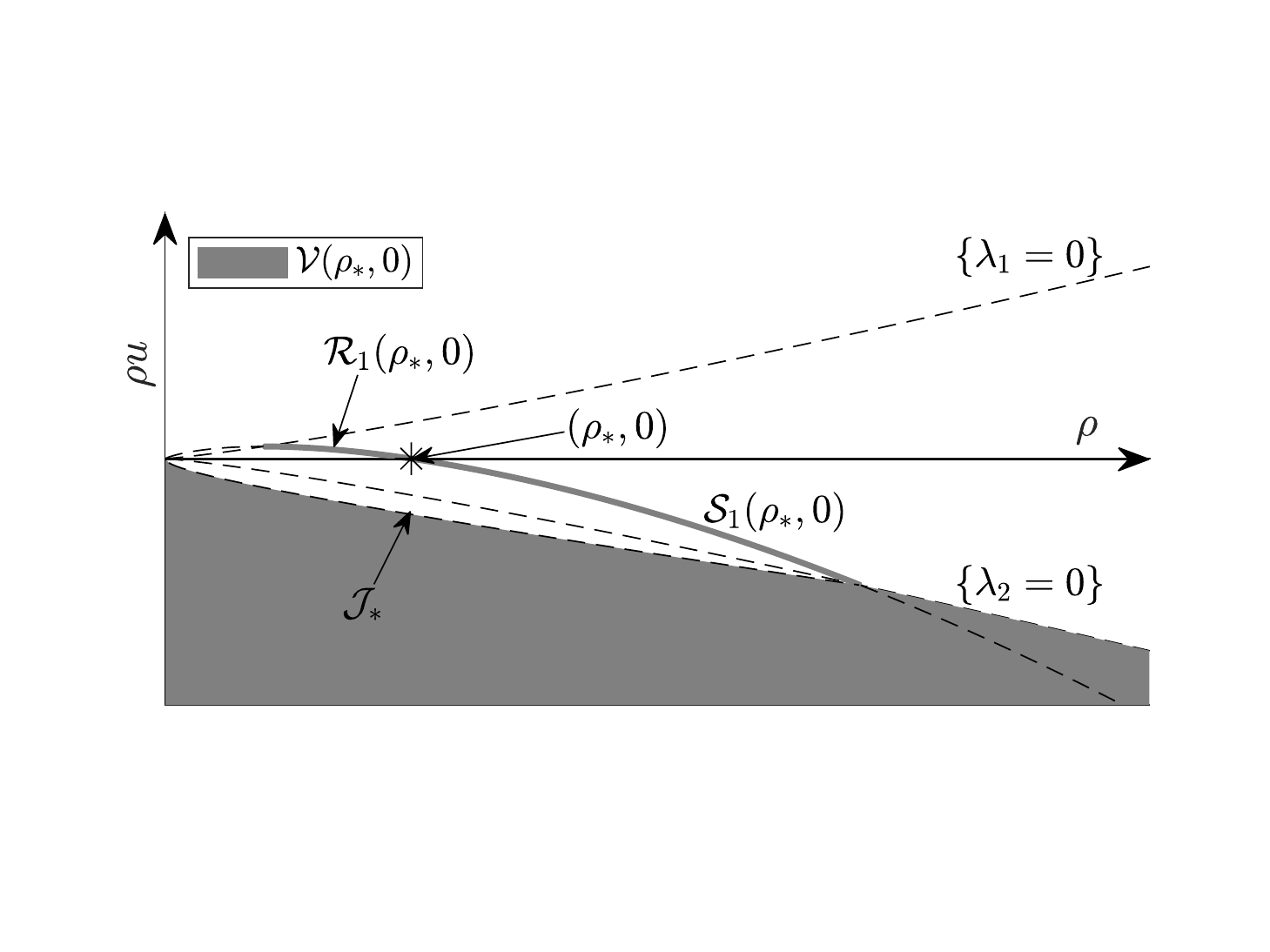}
\caption{The boundary Riemann set $\mathcal V(\rho_*,0)$}\label{fig:BoundaryRiemann}
\end{figure}
%
%
%
%
%

\section{Motivation and formal derivation}
\label{sec:MotivationandDerivation}
In this section we give a physical motivation and formal derivation for the new coupling condition. Both are based on the kinetic model for isentropic gas and a maximum energy/entropy dissipation principle. First, we specify the kinetic coupling condition which conserves mass and dissipates as much energy as possible. In the second step, we consider the macroscopic limit of the kinetic coupling condition. This relaxation works only on a formal level since the currently available results for passing to the limit at the junction are not strong enough. Nevertheless, we are able to take the formal limit towards the macroscopic coupling condition. Finally, we also give an interpretation of the resulting conditions. 

\subsection{A maximum energy dissipation principle applied to kinetic coupling conditions}
Since Dafermos \cite{Da1973} introduced the entropy rate admissibility criterion it is a natural approach to maximize the entropy dissipation in the field of hyperbolic conservation laws. This technique can be used to single out the physically correct solutions. We adapt this approach and aim to find the most dissipative kinetic coupling condition (with the constrained of conservation of mass). Since the physical energy is an entropy for the system of isentropic gas and is the physically relevant entropy, we maximize the energy dissipation.\\
We consider the kinetic BGK model of the isentropic gas equations which is given by
\begin{equation}
\partial_t f_\epsilon+ \xi\partial_x f_\epsilon=\frac{M[f_\epsilon]-f_\epsilon}{\epsilon},\quad \text{for a.e. } t>0,x\in\mathbb R,\xi\in\mathbb R,
\end{equation}
where $f_\epsilon=f_\epsilon(t,x,\xi)\in D$. 
\begin{remark}
The BGK model, its relaxation limit and boundary conditions were studied by Berthelin and Bouchut \cite{BeBo2000,BeBo2002Kinetic,BeBo2002Boundary}. These results were extended to networks in \cite{Ho2020}.  Notice that our considerations are independent of the right hand side of the kinetic equation as long as the kinetic solution converges to an entropy solution of the macroscopic equation.
\end{remark}
To couple the half-space solutions, we have to define a kinetic coupling condition
\begin{equation}
\Psi\colon L^1_\mu((-\infty,0)_\xi,D)^d\to L^1_\mu((0,\infty)_\xi,D)^d;\quad g\mapsto\Psi[g].
\end{equation}
The half-space solutions $f_\epsilon^k\colon(0,\infty)_t\times(0,\infty)_x\times\mathbb R_\xi\to D$ to the kinetic model are coupled by
\begin{equation}
\Psi^k[f_\epsilon^1(t,0,\cdot),\dots,f_\epsilon^d(t,0,\cdot)](\xi)=f_\epsilon^k(t,0,\xi), \quad\text{for a.e. }t>0, \xi>0.
\end{equation}
We are interested in kinetic coupling conditions which conserve mass. More precisely, we require that
\begin{align}\label{eq:ConservationMassKinetic}
\sum_{k=1}^d A_k \left(\int_0^\infty \xi \Psi^k_0[g](\xi)\,\mathrm{d}\xi+\int_{-\infty}^0 \xi g^k_0(\xi)\,\mathrm{d}\xi\right)=0,
\end{align}
holds for all $g\in L^1_\mu((-\infty,0)_\xi,D)^d$. Our aim is to find the coupling condition which dissipates as much energy as possible and conserves mass. The energy dissipation at the junction is given by
\begin{equation}\label{eq:MinimizingEnergy}
\sum_{k=1}^d A_k \left(\int_0^\infty \xi H(\Psi^k[g](\xi),\xi)\,\mathrm{d}\xi+\int_{-\infty}^0 \xi H(g^k(\xi),\xi)\,\mathrm{d}\xi\right).
\end{equation}
To find the unique minimizer of this functional, we use the convexity of the kinetic energy. More precisely, we use the sub-differential inequality (see e.g. \cite{BeBo2002Kinetic})
\begin{equation}\label{eq:ConvexityKineticEntropyInequality}
H(g,\xi)\ge H(M(\rho,u,\xi),\xi)+\eta'(\rho,u)(g-M(\rho,u,\xi)),
\end{equation}
for every $g\in D$, $(\rho,u)\in D$, $\xi\in \mathbb R$. Applying the sub-differential inequality leads to
\begin{align}
&\sum_{k=1}^d A_k \int_0^\infty \xi H(\Psi^k[g](\xi),\xi)\,\mathrm{d}\xi\nonumber\\
\ge& \sum_{k=1}^d A_k  \int_0^\infty \xi H(M(\rho_*,0,\xi),\xi)\,\mathrm{d}\xi\nonumber\\
&+ \eta'(\rho_*,0)\left[ \sum_{k=1}^d A_k \left(\int_0^\infty \xi \Psi^k_0[g](\xi)\,\mathrm{d}\xi -\int_0^\infty \xi M(\rho_*,0,\xi)\,\mathrm{d}\xi\right)\right]\nonumber\\
=&\sum_{k=1}^d A_k  \int_0^\infty \xi H(M(\rho_*,0,\xi),\xi)\,\mathrm{d}\xi,\label{eq:HilfMaximumEntropyDissipation}
\end{align}
where $\rho_*\ge 0$ is uniquely defined by
\begin{equation}\label{eq:DefinitionRhoStarKineticCoupling}
\sum_{k=1}^d A_k \left(\int_0^\infty \xi M_0(\rho_*,0,\xi)\,\mathrm{d}\xi +\int_{-\infty}^0 \xi g^k_0(\xi)\,\mathrm{d}\xi\right)=0.
\end{equation}
Such an $\rho_*$ always exists since 
\begin{equation}
\int_0^\infty \xi M_0(\rho_*,0,\xi)\,\mathrm{d}\xi=c_{\gamma,\kappa}\left(a_\gamma \rho_*^\theta\right)^{\frac{\gamma+1}{\gamma-1}}\int_0^1 z\,(1-z^2)^\lambda\,\mathrm{d}z.
\end{equation}
The uniqueness follows from the strict convexity of $H$.
We obtain the following result.
\begin{theorem}
The unique kinetic coupling condition 
\begin{equation*}
\Psi\colon L^1_\mu((-\infty,0)_\xi,D)^d\to L^1_\mu((0,\infty)_\xi,D)^d;\quad g\mapsto\Psi[g]
\end{equation*}
which conserves mass (\ref{eq:ConservationMassKinetic}) and minimizes (\ref{eq:MinimizingEnergy}) is given by
\begin{equation}\label{eq:DefOptimalCoupling}
\Psi^k[g](\xi):=M(\rho_*,0,\xi),
\end{equation}
where $\rho_*\ge 0$ is defined by equation (\ref{eq:DefinitionRhoStarKineticCoupling}).
\end{theorem}
According to \cite{Ho2020}, weak solutions to the kinetic BGK model on networks exist under suitable conditions on the initial data.\medskip\\
Instead of minimizing the kinetic energy, other kinetic entropies could be considered. Notice that every kinetic entropy $H_S$ parametrized by a strictly convex, symmetric function $S\in C^1(\mathbb R,\mathbb R)$ leads to the kinetic coupling condition obtained in (\ref{eq:DefOptimalCoupling}). This can be proven by applying the sub-differential inequality for $H_S$ and using $G_S(\rho_*,0)=\partial_{\rho u}\eta_{S}(\rho_*,0)=0$.

\subsection{Formal relaxation limit}
We take the formal limit at the junction with $\Psi^k[g](\xi)=M(\rho_*,0,\xi)$. Since (\ref{eq:ConvexityKineticEntropyInequality}) holds for every convex kinetic entropy $H_S$, we get
\begin{align}
&\int_{\mathbb R}\xi H_S(f^k_\epsilon(t,0,\xi),\xi)\,\mathrm{d}\xi-G_S(\rho_*^\epsilon(t),0)\nonumber\\
&\quad-\eta_S'(\rho_*^\epsilon(t),0)\left(\int_{\mathbb R}\xi f^k_\epsilon(t,0,\xi)\,\mathrm{d}\xi-F(\rho_*^\epsilon(t),0)\right)\le 0,
\end{align}
for $t>0$, $k=1,\dots,d$, where $\rho_*^\epsilon(t)\ge 0$ is defined by (\ref{eq:DefinitionRhoStarKineticCoupling}) with $g^k(\xi)=f^k_\epsilon(t,0,\xi)$ and fixed $\epsilon>0$. Assuming that $\rho^\epsilon_*\to\rho_*$ converges strongly in $L^1_{\mathrm loc}$ as $\epsilon\to 0$ and using the arguments in \cite{Ho2020}, we get
\begin{equation}\label{eq:EntropyInequalitiyFormalLimitIsentropic}
\overline{G_S(\rho_k,u_k)}(t)-G_S(\rho_*(t),0)-\eta_S'(\rho_*(t),0)(\overline{F(\rho_k,u_k)}(t)-F(\rho_*(t),0))\le 0,
\end{equation}
for a.e. $t>0$, $k=1,\dots,d$ and every convex $S\in C^1(\mathbb R,\mathbb R)$. Notice that it is open if the strong limit $\rho_*^\epsilon(t)\to\rho_*(t)$ can be justified.
Furthermore, (\ref{eq:EntropyInequalitiyFormalLimitIsentropic}) is the entropy formulation of boundary conditions induced by $\mathcal E(\rho_*(t),0)$. It was proven in \cite{Ho2020}, that mass remains conserved at the junction after taking the limit. More precisely, we have
\begin{equation}
\sum_{k=1}^d A_k \overline{\rho_k u_k}(t)=0,\quad\text{a.e. }t>0,
\end{equation}
for the weak traces $\overline{\rho_k u_k}$ at $x=0$.\medskip\\
We summarize that after formally taking the (strong) limit, the traces at $x=0$ satisfy the entropy formulation of the boundary condition $\mathcal E(\rho_*(t),0)$ with boundary data $(\rho_*(t),0)$ and mass is conserved at the junction. Assuming that the stronger formulation $\mathcal V(\rho_*(t),0)$ of the boundary condition holds true, we obtain immediately Definition \ref{def:GeneralizedRiemann} and Definition \ref{def:GeneralizedCauchyProblem}.

\subsection{Interpretation of the macroscopic coupling condition and necessary physical properties}
In this subsection, we restrict ourselves to the generalized Riemann problem since it is a building block for the Cauchy problem.\medskip\\
The new macroscopic coupling condition is an implicit condition compared to the known coupling conditions in the literature. The idea of the new coupling condition is to assume the existence of left hand states of zero speed, independent of $k$ and such that mass is conserved. This is different to the known coupling conditions which are based on a coupling of traces of physical quantities.\medskip\\
We made a particular choice by choosing $u_*=0$ for all left states. This choice can be interpreted in the following way. On the kinetic level the particles are stopped immediately after arriving at the junction. Then, the particles are instantaneously redistributed equally and into all pipelines. This artificial process leads to a coupling condition which does not prefer any pipeline and ignores the momentum of the incoming particles.\medskip\\
We interpret the macroscopic coupling condition by gas being stopped at the junction. Therefore, it is reasonable that we state a relation between the traces of the stagnation enthalpy at the junction. The stagnation enthalpy determines the enthalpy at a stagnation point after the gas is brought to a stop. This relation is given by inequalities depending on the signs of $\bar{u}_k$ at the junction (see Corollary \ref{corr:NonIncreasingEnergy}). Furthermore, the macroscopic coupling satisfies several properties which seem to be necessary for a physically correct coupling condition. These properties are non-increasing energy at the junction and a maximum principle on the Riemann invariants (see Section \ref{sec:PropertiesCoupling} for more details). Furthermore, the same derivation technique applied to the full Euler equations leads to very similar results (see Section \ref{sec:FullGasDynamics}).\medskip\\
We emphasis that the coupling condition does not coincide with the Rankine-Hugoniot conditions in the case $d=2$. This can be easily checked since momentum is not conserved at the junction. Notice that this fact is not a disadvantage of the coupling condition since we want to model the coupling condition with maximum energy dissipation and conservation of mass but we neglect conservation of momentum. An interpretation of our coupling condition for $d=2$ can be given by an infinitesimal small point were turbulence occurs due to a geometric effect at the junction. \medskip \\
Summarizing, the derivation of the coupling condition by the kinetic model and the maximum energy dissipation principle at the junction lead to a choice of an artificial state of zero speed at the junction. Furthermore, the interpretation of the macroscopic coupling condition by particles stopped at the junction and redistributed is only possible with a state of zero speed. From a formal mathematical point of view, the proofs of the physical properties in Section \ref{sec:PropertiesCoupling} work only if the artificial state has zero speed. This observation is due to the structure of (\ref{eq:EntropyInequalitiyFormalLimitIsentropic}) and the fact that $\eta'_{S,\rho u}(\rho_*,0)=0$ for symmetric $S$.\\

\section{The generalized Riemann problem} 
\label{sec:RiemannProblem}

In this section, we prove existence and uniqueness of solutions to the generalized Riemann problem. 
Our strategy is as follows. Due to Lemma \ref{lemma:ExistenceFixedRho}, there exists a solution for a fixed artificial density $\tilde\rho_*$ but possibly without conservation of mass at the junction. We define the mass production at the junction as a function of the artificial density $\tilde\rho_*$ and prove its continuity and monotonicity. We conclude with the intermediate value theorem. The proof is similar to the proof in \cite{Re2015}. Notice that the artificial density is a monotone momentum related coupling constant in the sense of \cite{Re2015}. The structure of the generalized Riemann problem allows us to extend the result to the supersonic region.

\begin{proposition}\label{prop:MomentumContinuous}
Assume that initial data $(\hat\rho_k,\hat u_k)\in D$ are given. Let $(\rho_k,u_k)$ be the function obtained in Lemma \ref{lemma:ExistenceFixedRho} with artificial density $\tilde\rho_*\ge 0$. Then, the trace $\bar\rho_k\bar u_k=(\rho_k u_k)(t,0+)$ is continuous with respect to $\tilde\rho_*$.
\end{proposition}
\begin{proof}
First, we prove the continuity with respect to the artificial density at fixed $\tilde\rho_*>0$. Let $\mathcal A_*,\mathcal B_*,\mathcal C_*,\mathcal J_*$ be as in Definition \ref{def:SetsBoundaryValueProb}.
If $(\hat\rho_k,\hat u_k)$ lies in the interior of $\mathcal A_*$, $\mathcal B_*$ and $\mathcal C_*$, the continuity follows from the fact that the wave curves and the curves defined by $\lambda_i(\rho,u)=0$, $i=1,2$ are continuous. Therefore, it remains to prove continuity at the boundaries.\\
\textit{Step 1:} First, we consider the boundary between $\mathcal A_*$ and $\mathcal B_*$. More precisely, the 2-wave curve $\mathcal W_2(\rho_\alpha,u_\alpha)$ with $\{(\rho_\alpha,u_\alpha)\}=\{\lambda_1=0\}\cap \mathcal R_1(\tilde\rho_*,0)$. We have
\begin{equation}
\lim_{\rho\searrow\tilde\rho_*}(\bar\rho_k, \bar u_k)=(\rho_\alpha, u_\alpha) = \lim_{\rho\nearrow\tilde\rho_*}(\bar\rho_k, \bar u_k),
\end{equation} 
since the 2-wave curve, the 1-rarefaction curve and the curve defined by $\{\lambda_1=0\}$ are continuous.\\
\textit{Step 2:} Next, we consider the boundary between $\mathcal B_*$ and $\mathcal C_*$. The continuity along $\mathcal R_2(\rho_\beta,u_\beta)$  is trivial since all involved curves are continuous. It remains to prove the continuity on $\mathcal J_*$. For $\rho>\tilde\rho_*$, we have
\begin{equation}\label{eq:lemmacontinuityproof1}
(\bar\rho_k, \bar u_k)(\rho)=(\hat\rho_{k}, \hat u_{k}),
\end{equation} 
where $(\bar\rho_k, \bar u_k)(\rho)$ is the trace at $x=0$ of the function obtained in Lemma \ref{lemma:ExistenceFixedRho} with artificial density $\rho$.
For $\rho<\tilde\rho_*$ and $|\rho-\tilde\rho_*|$ sufficiently small, the state $(\hat\rho_{k}, \hat u_{k})$ is connected to the boundary state $(\bar\rho_k,\bar u_k)(\rho)\in\mathcal S_1(\tilde\rho_*,0)$ by a 2-shock with (small) positive speed. We get
\begin{equation}\label{eq:lemmacontinuityproof2}
\lim_{\rho\nearrow\tilde\rho_*}\bar\rho_k \bar u_k=\hat\rho_{k} \hat u_{k},
\end{equation}
since the speed of the 2-shock tends to zero as $\rho\nearrow\tilde\rho_*$.
The continuity on $\mathcal J_*$ follows from (\ref{eq:lemmacontinuityproof1} -- \ref{eq:lemmacontinuityproof2}). Notice, that $\bar\rho_k$ and $\bar u_k$ itself are not continuous on $\mathcal J_*$.\\
The continuity at $\tilde\rho_*=0$ follows from similar considerations, see also \cite{LiSm1980}.
\end{proof}

\begin{lemma}[{\cite[Remark 1]{Re2015}}]
Along the reversed 2-wave curves monotonicity in $\rho_0$ is equivalent to monotonicity in $u_0$. More precisely,
\begin{align}
\ddfrac{u_0}{\rho_0}\Big|_{\mathcal W_2^-}>0,\quad\text{for all }(\rho,\rho u)\in D\backslash\{0\}.
\end{align}
The subscript denotes differentiation along the reversed $2$-wave curve $\mathcal W^-_2(\rho,u)$.
\end{lemma}
\begin{proof}
By the formula for the reversed 2-rarefaction wave, we have
\begin{equation*}
\ddfrac{u_0}{\rho_0}\Big|_{\mathcal R_2^-}=\sqrt{\kappa\gamma}\rho_0^{\frac{\gamma-3}{2}}>0.
\end{equation*}
Along the reversed 2-shock curve, we get
\begin{align*}
\ddfrac{u_0}{\rho_0}\Big|_{\mathcal S^-_2}&=\frac{\kappa}{2\rho_0\rho} \frac{(1-\gamma)\rho_0^{\gamma-1}\rho+\gamma\rho_0^\gamma-\frac{\rho^{\gamma+1}}{\rho_0}}{\sqrt{\frac{\kappa(\rho_0^\gamma-\rho^\gamma)(\rho_0-\rho)}{\rho_0\rho}}}\\
&=\frac{\kappa}{2\rho_0\rho (u_0-u)}\left(\frac{\rho}{\rho_0}(\rho_0^\gamma-\rho^\gamma)+\gamma\rho_0^{\gamma-1}(\rho_0-\rho)\right)>0,
\end{align*}
since $\rho<\rho_0$ and $u<u_0$.
\end{proof}

In the subsonic regime, we can determine the artificial density by a function $\tilde\rho_*=R_*(\bar\rho_k,\bar u_k)$.
\begin{definition}\label{def:PStar}
Let $R_*\colon D\to (0,\infty)$ be defined by
\begin{align}
R_*(\rho,u)=
\begin{cases}
\left(\rho^\theta+\frac{u}{a_\gamma}\right)^{1/\theta},\quad &\text{if }u\ge 0,\\
R_*,\quad  \text{with }u=-\sqrt{\frac{\kappa(\rho^\gamma-R_*^\gamma)(\rho-R_*)}{\rho R_*}},R_*<\rho,\quad &\text{if }u<0.
\end{cases}
\end{align}
\end{definition}
Notice that $R_*$ is well-defined, since for fixed $\rho>0$, the function 
\begin{equation*}
(0,\rho]\to(-\infty,0],\quad R_*\mapsto -\sqrt{\frac{\kappa(\rho^\gamma-R_*^\gamma)(\rho-R_*)}{\rho R_*}},
\end{equation*}
is bijective. By this definition we can reformulate \textit{RP1} in the subsonic regime by
\begin{equation}\label{eq:EqualArtificalDensity}
R_*(\bar\rho_k,\bar u_k)(t)=\mathcal H_{R_*}(t),\quad \text{for } k=1,\dots,d,\quad\text{for a.e. }t>0. 
\end{equation}
Compare the new condition (\ref{eq:EqualArtificalDensity}) with the coupling conditions in (\ref{eq:defEqualPressure} -- \ref{eq:defEqualStagnationEnthalpy}) and note that they are different even in the subsonic regime.
\begin{lemma}\label{lemma:PStaralong2Wave}
We have
\begin{align}
\ddfrac{R_*}{\rho_0}\Big|_{\mathcal W_2^-}>0,\quad\text{for all }(\rho, u)\in D\backslash\{0\}.
\end{align}
The subscript represents the differentiation along the reversed 2-wave curve.
\end{lemma}
\begin{proof}
For $u>0$, differentiation along the 2-wave curve gives
\begin{align}
\ddfrac{R_*}{\rho_0}&=\left(\rho_0^{\frac{\gamma-3}{2}}+\frac{1}{\sqrt{\kappa\gamma}}\ddfrac{u_0}{\rho_0}\right) \left(\rho_0^\theta+\frac{u_0}{a_\gamma}\right)^{\frac{3-\gamma}{\gamma-1}}>0,\label{eq:proofHilfLemma3}
\end{align}
since $\ddfrac{u_0}{\rho_0}>0$.
For $u<0$, differentiation of
\begin{equation*}
u_0=-\sqrt{\frac{\kappa(\rho_0^\gamma-R_*^\gamma)(\rho_0-R_*)}{\rho_0 R_*}}
\end{equation*}
along the 2-wave curve gives
\begin{equation*}
\ddfrac{u_0}{\rho_0}=\frac{\kappa}{u_0}\left[\gamma\left(\rho_0^{\gamma-1}-R_*^{\gamma-1}\ddfrac{R_*}{\rho_0}\right)\left(\frac{1}{R_*}-\frac{1}{\rho_0}\right)+(\rho_0^\gamma-R^\gamma_*)\left(\frac{1}{\rho_0^2}-\frac{1}{R_*^2}\ddfrac{R_*}{\rho_0}\right)\right],
\end{equation*}
or equivalently
\begin{align*}
\ddfrac{R_*}{\rho_0}=\frac{\frac{2 u_0}{\kappa}\ddfrac{u_0}{\rho_0}+\gamma\rho_0^{\gamma-1}\left(\frac{1}{\rho_0}-\frac{1}{R_*}\right)+\frac{1}{\rho_0^2}(R_*^\gamma-\rho_0^\gamma)}{\gamma R_*^{\gamma-1}\left(\frac{1}{\rho_0}-\frac{1}{R_*}\right)+\frac{1}{R^2_*}(R_*^\gamma-\rho_0^\gamma)}.
\end{align*}
The right hand side is strictly positive, since $\rho_0>R_*$, $u_0<0$ and $\ddfrac{u_0}{\rho_0}>0$.
\end{proof}

\begin{lemma}\label{lemma:PStaralong2Eigenvalue}
We have
\begin{equation*}
\ddfrac{R_*}{\rho}\Big|_{\lambda_1=0}>0,
\end{equation*}
where the subscript denotes the differentiation along the curve $\{\lambda_1=0\}$.
\end{lemma}
\begin{proof}
We differentiate $u$ along the curve defined by $\{\lambda_1=0\}$ and get
\begin{equation*}
\ddfrac{u}{\rho}=\ddfrac{}{\rho}\left(\sqrt{\kappa\gamma}\rho^\theta\right)=\sqrt{\kappa\gamma}\theta\rho^{\frac{\gamma-3}{2}}>0.
\end{equation*}
The result follows with (\ref{eq:proofHilfLemma3}).
\end{proof}

\begin{proposition}\label{prop:MomentumIncreasing}
Fix initial data $(\hat\rho_k,\hat u_k)\in D$. Let $(\rho_k,u_k)$ be the solution obtained in Lemma \ref{lemma:ExistenceFixedRho} with artificial density $\tilde\rho_*\ge 0$. Then, the trace of the momentum $\bar\rho_k\bar u_k$ is increasing in $\tilde\rho_*$. It is strictly increasing if $(\hat\rho_k,\hat u_k)\in \mathcal A_* \cup\mathcal B_*$, $\tilde\rho_*>0$ and constant if $(\hat\rho_k,\hat u_k)\in \operatorname{int}\mathcal C_*$.
\end{proposition}
\begin{proof}
\textit{Step 1:} We consider the case $(\hat\rho_k,\hat u_k)\in \operatorname{int}\mathcal A_*$, $\tilde\rho_*>0$. We already know that $(\bar\rho_k,\bar u_k)\in\mathcal R_1(\tilde\rho_*,0)\cap\{\lambda_1=0\}$. The formulas for $\mathcal R_1(\tilde\rho_*,0)$ and $\lambda_1$ lead to
\begin{equation*}
\frac{\mathrm{d}(\bar\rho_k \bar u_k)}{\mathrm{d} \tilde\rho_*}
=\frac{\mathrm d}{\mathrm{d}\tilde\rho_*}\left(\sqrt{\gamma\kappa}\left(\frac{2}{\gamma+1} \right)^{\frac{\gamma+1}{\gamma-1}}\tilde\rho_*^{\frac{\gamma+1}{2}}\right)
=\sqrt{\kappa\gamma}\left(\frac{2}{\gamma+1}\right)^{\frac{2}{\gamma-1}}\tilde\rho_*^{\frac{\gamma-1}{2}}>0.
\end{equation*}
\textit{Step 2:} Next, we consider the case $(\hat\rho_k,\hat u_k)\in \mathcal B_*$, $\tilde\rho_*>0$. $(\bar\rho_k,\bar u_k)$ is the unique element in $\mathcal W_2(\hat\rho_k,\hat u_k)\cap \mathcal W_1(\tilde\rho_*,0)$. Since Lemma \ref{lemma:PStaralong2Wave}, we have 
\begin{equation*}
\frac{\mathrm{d} R_*}{\mathrm{d}\rho_0}\Big|_{\mathcal W_2^-}>0,
\end{equation*}
but this implies 
\begin{equation*}
\ddfrac{\bar\rho_k}{\tilde\rho_*}=\ddfrac{R_*^{-1}(\tilde\rho_*)}{\tilde\rho_*}>0,
\end{equation*}
where $R_*^{-1}(\tilde\rho_*)$ is the unique element in $\mathcal W_1(\tilde\rho_*,0)\cap \mathcal W_2(\hat\rho_k,\hat u_k)$. By chain rule, we get
\begin{equation*}
\ddfrac{(\bar\rho_k \bar u_k)}{\tilde\rho_*}=\ddfrac{(\bar\rho_k \bar u_k)}{\bar\rho_k}\cdot \ddfrac{\bar\rho_k}{\tilde\rho_*}.
\end{equation*}
Therefore, it remains to prove 
\begin{equation*}
\ddfrac{(\bar\rho_k \bar u_k)}{\bar\rho_k}\Big|_{\mathcal W_2^-}>0.
\end{equation*}
On $\mathcal R_2^-(\hat \rho_k,\hat u_k)$, we have
\begin{align*}
\ddfrac{(\bar\rho_k \bar u_k)}{\bar\rho_k}=&\bar u_k+\bar \rho_k\ddfrac{\bar u_k}{\bar\rho_k}=\bar u_k+\sqrt{\kappa\gamma}\bar \rho_k^{\theta}=\lambda_2(\bar \rho_k,\bar u_k)>0,
\end{align*}
since $(\bar \rho_k,\bar u_k)\in\mathcal B_*$ and Lemma \ref{lemma:ExistenceFixedRho}.
On $\mathcal S_2^-(\hat \rho_k,\hat u_k)$, we use the fact that $\lambda_2(\bar\rho_k,\bar u_k)>0$ and get
\begin{align*}
\ddfrac{(\bar\rho_k \bar u_k)}{\bar\rho_k}&=\bar u_k+\bar \rho_k\ddfrac{\bar u_k}{\bar\rho_k}\\
&=\bar u_k+\frac{\sqrt{\kappa}}{2\bar\rho_k \hat\rho_k(\bar u_k-\hat u_k)}\left[\gamma\bar\rho_k^\gamma(\bar\rho_k-\hat\rho_k)+\hat\rho_k(\bar\rho_k^\gamma-\hat\rho_k^\gamma) \right]\\
&=\lambda_2(\bar\rho_k,\bar u_k)-\sqrt{\kappa\gamma}\bar\rho_k^\theta+\frac{\sqrt{\kappa}}{2\bar\rho_k \hat\rho_k(\bar u_k-\hat u_k)}\big[\gamma\bar\rho_k^\gamma(\bar\rho_k-\hat\rho_k)+\hat\rho_k(\bar\rho_k^\gamma-\hat\rho_k^\gamma) \big]\\
&=\lambda_2(\bar\rho_k,\bar u_k)+\frac{\sqrt{\kappa}}{2\bar\rho_k \hat\rho_k(\bar u_k-\hat u_k)}\\
&\cdot\left[ \gamma\bar\rho_k^\gamma(\bar\rho_k-\hat\rho_k)+\hat\rho_k(\bar\rho_k^\gamma-\hat\rho_k^\gamma)-2\sqrt{\gamma\bar\rho_k^\gamma(\bar\rho_k-\hat\rho_k)}\sqrt{\hat\rho_k (\bar\rho_k^\gamma-\hat\rho_k^\gamma)}\right]\\
&=\lambda_2(\bar\rho_k,\bar u_k)+\frac{\sqrt{\kappa}}{2\bar\rho_k \hat\rho_k(\bar u_k-\hat u_k)} \left[\sqrt{\gamma\bar\rho_k^\gamma(\bar\rho_k-\hat\rho_k)}-\sqrt{\hat\rho_k(\bar\rho_k^\gamma-\hat\rho_k^\gamma)}\right]^2\\
&>\lambda_2(\bar\rho_k,\bar u_k)> 0,
\end{align*}
for $\bar\rho_k\neq\hat\rho_k$. For $(\bar \rho_k,\bar u_k)\in \mathcal W_1(\tilde\rho_*,0)$ the left and right limit of the derivatives along $\mathcal W^-_2(\bar \rho_k,\bar u_k)$ are strictly positive and strict monotonicity in $\mathcal B_*$ follows.\\
\textit{Step 3:} The strict monotonicity at the boundary between $\mathcal A_*$ and $\mathcal B_*$ can be shown by taking the left and right limit of the derivatives which are strictly positive.\\
\textit{Step 4:} Finally, we consider the case $(\hat\rho_k,\hat u_k)\in \mathcal C_*$. In the case $\lambda_2(\hat\rho_k,\hat u_k)\le 0$, we have $(\bar \rho_k,\bar u_k)=(\hat \rho_k,\hat u_k)$ and observe that $(\bar \rho_k,\bar u_k)$ is locally constant as a function of $\tilde\rho_*$.
For $\lambda_2(\hat\rho_k,\hat u_k)\ge 0$, we have
\begin{equation*}
\{(\bar\rho_k,\bar u_k)\}= \mathcal R_2(\hat\rho_k,\hat u_k)\cap \{\lambda_2 =0\}.
\end{equation*}
Therefore, $(\bar\rho_k,\bar u_k)$ is locally constant with respect to $\tilde\rho_*$.
\end{proof}

\begin{theorem}
Assume that the initial states $(\hat\rho_k,\hat u_k)\in D$ are given. Then, there exists a unique solution $(\rho_k,u_k)$ to the generalized Riemann problem according to Definition \ref{def:GeneralizedRiemann} with a unique artificial density $\rho_*\ge 0$.
\end{theorem}
\begin{proof}
For the solution $(\rho_k,u_k)$ obtained in Lemma \ref{lemma:ExistenceFixedRho} with artificial density $\tilde\rho_*\ge 0$, the mass production at the junction is given by 
\begin{equation*}
m(\tilde\rho_*)=\sum_{k=1}^d A_k \bar\rho_k \bar u_k.
\end{equation*}
Since Proposition \ref{prop:MomentumContinuous} and \ref{prop:MomentumIncreasing}, the function $m$ is continuous and increasing in $\tilde\rho_*$. We will use these properties to prove existence and uniqueness by the intermediate value theorem. We divide the rest of the proof in two steps.\\
\textit{Step 1:} We prove that there exist $0\le \rho_-<\rho_+$ such that
\begin{equation*}
m(\rho_-)\le 0\le m(\rho_+).
\end{equation*}
\begin{itemize}
	\item We set
	\begin{equation*}
	\rho_-=\operatorname{argmin}\;\{\rho\ge 0\,|\,(\rho,0)\in \mathcal W_2^-(\hat\rho_k,\hat u_k)\text{ for }k=1,\dots,d\}.
	\end{equation*}
	Then, we have $(\bar\rho_k,\bar u_k)\in\mathcal V(\rho_-,0)\cap \{\rho u\le 0\}$ for all $k=1,\dots,d$, but this implies $m(\rho_-)\le 0$.
	\item We set
	\begin{equation*}
	\rho_+=\operatorname{argmax}\;\{\rho\ge 0\,|\,(\rho,0)\in \mathcal W_2^-(\hat\rho_k,\hat u_k)\text{ for }k=1,\dots,d\}.
	\end{equation*}
	Then, we have $(\bar\rho_k,\bar u_k)\in \mathcal V(\rho_+,0)\cap \{\rho u\ge 0\}$ for all $k=1,\dots,d$, but this implies $m(\rho_+)\ge 0$.
\end{itemize}
By the intermediate value theorem, we conclude that there exists $\rho_*\ge 0$ such that $m(\rho_*)=0$.\\
\textit{Step 2:} We prove that $m$ is strictly increasing at $\rho_*$. Since $m(\rho_*)=0$, there exists $1\le k_0\le d$ such that $\bar\rho_{k_0}\bar u_{k_0}\ge 0$. Due to Proposition \ref{prop:MomentumIncreasing} and $(\bar\rho_{k_0},\bar u_{k_0})\in\mathcal A_*\cup\mathcal B_*$, $\bar\rho_{k_0}\bar u_{k_0}$ is strictly increasing with respect to $\tilde\rho_*$. Thus, $m$ is strictly increasing at $\rho_*$. This implies the uniqueness of the artificial density $\rho_*\ge 0$ with $m(\rho_*)=0$.\\
Since Riemann problems admit unique self-similar Lax solutions and $\rho_*$ is uniquely defined by $(\bar\rho_k,\bar u_k)$ with $\bar\rho_k \bar u_k\ge 0$, the solution to the generalized Riemann problem is unique.
\end{proof}

\section{The generalized Cauchy problem}
\label{sec:CauchyProblem}
In this section, we prove existence and uniqueness of solutions to the Cauchy problem. This result is based on a general existence theorem by Colombo, Herty and Sachers \cite{CHS2008} and holds true in a neighborhood of a subsonic solution. We also obtain Lipschitz continuous dependence on the initial data.

\begin{definition}\label{def:GeneralizedCauchyProblem}
Fix $(\hat\rho_1,\hat\rho_1\hat u_1,\dots,\hat\rho_d,\hat\rho_d\hat u_d)\in U^0+L^1(0,\infty)_x$ with $(\hat\rho_k,\hat\rho_k\hat u_k)\in D$ a.e. and $T\in(0,\infty]$. Then,  we call $(\rho_1,\rho_1 u_1,\dots,\rho_d,\rho_d u_d)\in C([0,T]_t,U^0+L^1(0,\infty)_x)$ a weak solution to the generalized Cauchy problem if $(\rho_k,\rho_k u_k)$, $k=1,\dots,d$ are weak entropy solutions to the isentropic gas equations and the following assertions hold true:
\begin{itemize}
		\item [CP0:] The solution satisfies the initial condition
			\begin{equation*}
				(\rho_k,u_k)(0+,x)=(\hat{\rho}_k,\hat{u}_k)(x)\in D,
			\end{equation*}
			 for a.e. $x>0,\,k=1,\dots,d$.
    \item [CP1:] For a.e. $t>0$, there exists $\rho_*(t)\ge 0$ such that 
			\begin{equation*}
				(\bar\rho_k,\bar u_k)(t)\in\mathcal V (\rho_*(t),0)
			\end{equation*}
			for all $k=1,\dots,d$.
    \item [CP2:] Mass is conserved at the junction
			\begin{equation*}
				\sum_{k=1}^d A_k\bar{\rho}_k\bar{u}_k=0,\quad\text{for a.e. }t>0.
			\end{equation*}
\end{itemize}
\end{definition}

\begin{theorem}
Fix a vector of subsonic states $U^0=( \rho^0_1,\rho_1^0 u^0_1,\dots,\rho^0_d,\rho_d^0 u_d^0)\in D^d$ such that the corresponding generalized Riemann problem admits a stationary solution. Then, there exist $\delta, L>0$ and a map $S\colon [0,\infty)\times \mathcal D\to\mathcal D$ such that
\begin{itemize}
	\item $\mathcal D \supset \{ U\in  U^0+L^1((0,\infty)_x,D^d)\, ;\, TV(U)\le \delta \}$;
	\item For $U\in \mathcal D$, $S_0 U=U$ and for $s,t\ge 0$, $S_s S_t U=S_{s+t}U$;
	\item For $U,V\in \mathcal D$ and $s,t\ge 0$, $\| S_t U-S_s V\|_{L^1}\le L (\|U-V\|_{L^1}+|t-s|)$;
	\item If $U\in\mathcal D$ piecewise constant, then for $t>0$ sufficiently small, $S_t U$ coincides with the juxtaposition of the solution to Riemann problems centered at the points of jumps or at the junction.
\end{itemize}
Moreover, for every $U\in\mathcal D$, the map $t\mapsto S_t U$ is a solution to the generalized Cauchy problem.
\end{theorem}
\begin{proof}
Since $(\hat\rho_k,\hat u_k)$ is subsonic, we can choose $\delta>0$ sufficiently small such that $\mathcal D$ is contained in the subsonic region. Therefore, \textit{CP1} is equivalent to
\begin{equation*}
R_*(\bar\rho_k,\bar u_k)(t)=\mathcal H_{R_*}(t),\quad \text{for } k=1,\dots ,d,
\end{equation*}
for a.e. $t>0$. Note, that $R_*$ is defined in Definition \ref{def:PStar}.
Next, we apply Theorem 3.2 in \cite{CHS2008}. Therefore, we define the function
\begin{equation*}
\Psi(U)=
\begin{pmatrix}
\sum_{k=1}^d A_k \rho_k u_k\\
R_*(\rho_1,u_1)-R_*(\rho_2,u_2)\\
\vdots\\
R_*(\rho_{d-1},u_{d-1})-R_*(\rho_d,u_d)
\end{pmatrix}.
\end{equation*}
It remains to prove the transversality condition
\begin{equation}\label{eq:TransversalityCauchy}
\operatorname{det}\begin{bmatrix}D_1\Psi( U^0)\cdot r_2(\rho_1^0,u_1^0) &\cdots& D_d\Psi( U^0)\cdot r_d(\rho_d^0,u_d^0)\end{bmatrix}\neq 0,
\end{equation}
where $D_k=D_{(\rho_k,\rho_k u_k)}$.
By Lemma \ref{lemma:PStaralong2Wave} and the proof of Proposition \ref{prop:MomentumIncreasing}, we get
\begin{align*}
\frac{\partial \rho u}{\partial \rho}\Big|_{\mathcal W^-_2}>0\quad \text{and} \quad\frac{\partial R_*}{\partial \rho}\Big|_{\mathcal W^-_2}>0
\end{align*}
in the subsonic region. We deduce that
\begin{align*}
D_k(\rho_k u_k)\cdot r_2(\rho_k ,u_k)>0\quad\text{and}\quad D_k R_*(\rho_k, u_k)\cdot r_2(\rho_k, u_k)>0.
\end{align*}
This implies that the matrix involved in (\ref{eq:TransversalityCauchy}) has components with fixed sign which are given by
\begin{equation*}\label{eq:SignsCellsMatrixTransversality}
\begin{pmatrix}
+  & +     & +     &\cdots &    \cdots      & +  \\
+  & -     & 0      & \cdots &          \cdots   & 0  \\
0  & + & - & \ddots &       &\vdots  \\
\vdots & \ddots & \ddots & \ddots & \ddots&\vdots   \\
\vdots &  & \ddots & \ddots & \ddots&0   \\
0   &\dots& \dots & 0 & + & -   
\end{pmatrix}.
\end{equation*}
A Laplace expansion implies that the determinant of this matrix has a fixed sign and is non-zero.
\end{proof}

\begin{remark}
The existence and uniqueness result is restricted to subsonic initial data with sufficiently small total variation. The global result for the generalized Riemann problem and the large amount of inequalities for entropy fluxes at the junction (Propsition \ref{prop:EntropyDissipationJunction}) motivate to prove a more general result. Notice that the method in \cite{Ho2020} based on compensated compactness can be applied to the kinetic coupling condition (\ref{eq:DefOptimalCoupling}). This result justifies the relaxation in the interior of the pipelines. Nevertheless, it is open how the traces relax at the junction and if the obtained macroscopic solution satisfies the coupling condition in Definition \ref{def:GeneralizedCauchyProblem}.
\end{remark}
\section{Energy/Entropy dissipation at the junction}
\label{sec:PropertiesCoupling}
In this section, we prove some physical properties of the coupling condition. In particular, we prove that energy is non-increasing at the junction, a relation for the stagnation enthalpy and a maximum principle on the Riemann invariants.
\begin{proposition}\label{prop:EntropyDissipationJunction}
Assume that initial states $(\hat\rho_k,\hat u_k)\in D$ are given. Let $(\rho_k,u_k)$ be the solution to the generalized Riemann or Cauchy problem. Then, 
\begin{equation*}
\sum_{k=1}^d A_k G_S(\bar\rho_k,\bar u_k)\le 0,
\end{equation*}
for every convex $S\colon\mathbb R\to\mathbb R$ with $S(v)=S(-v)$.
\end{proposition}
\begin{proof}
The case $\rho_*=0$ is trivial. For $\rho_*\neq 0$ fix $S\in C^1(\mathbb R,\mathbb R)$ with $S(v)=S(-v)$. Since $(\bar\rho_k,\bar u_k)\in\mathcal V(\rho_*,0)\subset\mathcal E(\rho_*,0)$, we have 
\begin{equation*}
G_S(\bar\rho_k,\bar u_k)-G_S(\rho_*,0)-\eta_S'(\rho_*,0)\left(F(\bar\rho_k,\bar u_k)-F(\rho_*,0)\right)\le 0.
\end{equation*}
Notice that
\begin{align*}
G_S(\rho_*,0)&=\int_{\mathbb R}\theta v \chi(\rho_*,v)S(v)\,\mathrm{d}v=0, \quad \text{and}\\
\partial_{\rho u}\eta_{S}(\rho_*,0)&=\frac{1}{J_\lambda}\int_{-1}^1 (1-z^2)^\lambda S'(a_\gamma\rho^\theta_* z)\mathrm{d}z=0,
\end{align*}
since the integrands are anti-symmetric. These observations together with conservation of mass at the junction give 
\begin{align*}
0&\ge \sum_{k=1}^d A_k \left[ G_S(\bar\rho_k,\bar u_k)-G_S(\rho_*,0)-\eta_S'(\rho_*,0)(F(\bar \rho_k,\bar u_k)-F(\rho_*,0))\right]\\
&=\sum_{k=1}^d A_k  G_S(\bar\rho_k,\bar u_k)-\partial_\rho\eta_{S}(\rho_*,0)\left(\sum_{k=1}^d A_k\bar \rho_k\bar u_k\right)\\
&=\sum_{k=1}^d A_k  G_S(\bar\rho_k,\bar u_k).
\end{align*}
An approximation argument leads to the result for general convex functions $S$. 
\end{proof}

\begin{corollary}[Non-increasing energy]\label{corr:NonIncreasingEnergy}
Assume that sufficiently regular initial data $(\hat\rho_k,\hat u_k)$ are given. Let $(\rho_k,u_k)$ be the solution to the generalized Riemann or Cauchy problem. Then, the following properties hold true:
\begin{enumerate}[label=(\roman*)]
	\item At the junction energy is non-increasing, i.e.
		\begin{equation*}
			\sum_{k=1}^d A_k G(\bar\rho_k,\bar u_k)\le 0, \quad \text{for a.e. }t>0.
		\end{equation*}
	\item At the junction the traces of the stagnation enthalpy
		\begin{equation*}
			h(\rho,u)=\frac{u^2}{2}+\frac{\kappa\gamma}{\gamma-1}\rho^{\gamma-1}
		\end{equation*}
		are related by
		\begin{equation*}
			h(\bar\rho_k,\bar u_k)\le h(\rho_*,0)\le h(\bar\rho_l,\bar u_l),
		\end{equation*}
			for $\bar u_l\le 0\le \bar u_k, 1\le k,l\le d$, for a.e. $t>0$.
\end{enumerate}
\end{corollary}
\begin{proof}
Applying Proposition \ref{prop:EntropyDissipationJunction} to $S(v)=v^2/2$ gives
\begin{equation*}\label{eq:proofEnergyDissipationJunction}
	\sum_{k=1}^d A_k G(\bar\rho_k,\bar u_k)\le 0.
\end{equation*}
Since $(\rho_k,u_k)\in\mathcal V(\rho_*,0)\subset\mathcal E(\rho_*,0)$, we have
\begin{equation*}
G(\bar\rho_k,\bar u_k)-\partial_\rho \eta(\rho_*,0)\bar\rho_k\bar u_k\le 0.
\end{equation*}
The result follows from dividing the inequality by $\bar\rho_k\bar u_k\neq 0$ and the fact $\partial_\rho \eta(\rho_*,0)=h(\rho_*,0)$. The cases $\bar\rho_k\bar u_k=0$ and $\rho_*=0$ are trivial.
\end{proof}

\begin{corollary}[Maximum principle]
Let $(\rho_k,u_k)\colon(0,\infty)_t\times(0,\infty)_x\to D$ be the solution to the generalized Riemann or Cauchy problem with initial condition $(\hat\rho_k,\hat u_k)$. Assume that
\begin{equation*}
	-\omega_M\le\omega_1(\hat\rho_k,\hat u_k)(x)<\omega_2(\hat\rho_k,\hat u_k)(x)\le\omega_M,
\end{equation*}
for a.e. $x>0$, $k=1,\dots,d$. Then, we have 
\begin{equation*}
	-\omega_M\le\omega_1(\rho_k, u_k)(t,x)<\omega_2(\rho_k, u_k)(t,x)\le\omega_M,
\end{equation*}
for a.e. $t,x>0,\, k=1,\dots,d$.
\end{corollary}
\begin{proof}
We define the symmetric, positive function
\begin{equation*}
S_M(v)=(-\omega_M-v)^2_+ +(v-\omega_M)^2_+,\quad v\in\mathbb R.
\end{equation*}
As proven in \cite{Ho2020}, the definition of $\eta_{S_M}$ implies
\begin{equation}\label{eq:ProofRiemannInvariantsEntropyEquivalence}
\eta_{S_M}(\rho,u)=0,\text{ if and only if }-\omega_M\le\omega_1(\rho, u)<\omega_2(\rho, u)\le\omega_M.
\end{equation}
Furthermore, the divergence theorem and the entropy condition give
\begin{gather*}
\int_0^\infty \eta_{S_M}(\rho_k,u_k)(T,x)\,\mathrm{d}x-\int_0^\infty \eta_{S_M}(\rho_k,u_k)(0,x)\,\mathrm{d}x\\
-\int_0^T G_{S_M}(\rho_k,u_k)(t,0)\,\mathrm{d}t\le 0,
\end{gather*}
for $T>0$. Taking the sum over $k$, Proposition \ref{prop:EntropyDissipationJunction} and (\ref{eq:ProofRiemannInvariantsEntropyEquivalence}) lead to
\begin{align*}
0\le \sum_{k=1}^d A_k \int_0^\infty \eta_{S_M}(\rho_k,u_k)(T,x)\,\mathrm{d}x\le \sum_{k=1}^d A_k \int_0^\infty \eta_{S_M}(\rho_k,u_k)(0,x)\,\mathrm{d}x=0.
\end{align*}
The result follows from (\ref{eq:ProofRiemannInvariantsEntropyEquivalence}).
\end{proof}



\section{Numerical results}
\label{sec:NumericalResults}
We give an example in which the new coupling condition produces the physically correct wave types. This observation is based on the assumption of the appearance of turbulence at the junction. The other coupling conditions produce different wave types. Furthermore, we study level sets associated to unphysical coupling conditions. We will consider the coupling conditions with equal pressure, equal momentum flux, equal stagnation enthalpy and the artificial density coupling condition.

\subsection{A numerical example}
We consider the shallow water equations and set $\gamma = 2$, $\kappa=5$. We aim to compute solutions to the gerenalized Riemann problem in the sense of Lax. \medskip\\
The implementation is based on Newton's method applied to the coupling condition. The unknowns are the parameter of the reversed $2$-wave curve of the initial data. We will consider examples which can be solved in the subsonic region. Therefore, we can use the function $R_*$ (see Definition \ref{def:PStar}) to implement the artificial density coupling condition.\medskip\\
We consider a three junction situation with one ingoing and two outgoing pipelines. The idea of this example is to assume that the sum of the momentum of the initial states is zero and the initial densities coincide. For this example we clearly get a stationary solution if we take the coupling condition with equal pressure. To make the example more concrete, we take the initial data in Table \ref{tab:Ex1InitialData}.
\begin{table}
\begin{footnotesize}
\begin{center}
\begin{tabular}{|c|c|c|}
	\hline
	pipeline & $\hat\rho_{k}$ & $\hat\rho_{k} \hat u_{k}$ \\
	\hline
	1 & $+1.0000$ & $-1.0000$\\
	\hline
	2 & $+1.0000$ & $+0.5000$\\
	\hline
	3 & $+1.0000$ & $+0.5000$\\
	\hline
\end{tabular}
\end{center}
\end{footnotesize}
\caption{Initial data}
\label{tab:Ex1InitialData}
\end{table}
The traces and energy dissipation of the numerical solutions are given in Table \ref{tab:Ex1NumericalResults}.
\begin{table}
\begin{scriptsize}
\begin{center}
\begin{tabular}{|c|c|c|c|c|c|c|c|c|}
	\hline
	& \multicolumn{2}{c|}{\thead{Equal \\ density}}& \multicolumn{2}{c|}{\thead{Equal \\ momentum \\ flux}}& \multicolumn{2}{c|}{\thead{Equal \\ stagnation \\ enthalpy}}& \multicolumn{2}{c|}{\thead{Equal \\ artificial \\ density}}\\
	\hline
	pipeline & $\bar\rho_{k}$ & $\bar\rho_{k} \bar u_{k}$ & $\bar\rho_{k}$ & $\bar\rho_{k} \bar u_{k}$ & $\bar\rho_{k}$ & $\bar\rho_{k} \bar u_{k}$ & $\bar\rho_{k}$ & $\bar\rho_{k} \bar u_{k}$ \\
	\hline
	1 & $+1.0000$ & $-1.0000$ & $+0.8964$ & $-1.1981$ & $+0.8518$ & $-1.2670$ & $+1.1776$ & $-0.5417$\\
	\hline
	2 & $+1.0000$ & $+0.5000$ & $+1.0266$ & $+0.5991$ & $+1.0356$ & $+0.6335$ & $+0.9346$ & $+0.2708$\\
	\hline
	3 & $+1.0000$ & $+0.5000$ & $+1.0266$ & $+0.5991$ & $+1.0356$ & $+0.6335$ & $+0.9346$ & $+0.2708$\\
	\hline
	\thead{Energy \\ dissipation} 
	& \multicolumn{2}{c|}{$-7.5000\times 10^{-2}$}
	& \multicolumn{2}{c|}{$-1.725\times 10^{-2}$}
	& \multicolumn{2}{c|}{$\approx 0$}
	& \multicolumn{2}{c|}{$-1.3852\times 10^{-1}$}\\
	\hline
\end{tabular}
\end{center}
\end{scriptsize}
\caption{Numerical results}
\label{tab:Ex1NumericalResults}
\end{table}
The generalized Riemann problems are solved by the $2$-waves in Table \ref{tab:Ex1WaveTypes}.\medskip\\
\begin{table}
\begin{footnotesize}
\begin{center}
\begin{tabular}{|c|c|c|c|c|}
	\hline
	pipeline & \thead{Equal \\ density}
	& \thead{Equal \\ momentum flux}
	& \thead{Equal \\ stagnation enthalpy}
	& \thead{Equal \\ artificial density}\\
	\hline
	1 & no waves & rarefaction wave & rarefaction wave & shock\\
	\hline
	2 & no waves & shock & shock & rarefaction wave\\
	\hline
	3 & no waves & shock & shock & rarefaction wave\\
	\hline
\end{tabular}
\end{center}
\end{footnotesize}
\caption{Wave types}
\label{tab:Ex1WaveTypes}
\end{table}

\noindent
We make the following observations:
\begin{itemize}
	\item The $2$-waves types obtained by solving with the artificial density coupling condition are different to the other coupling conditions. The wave types of the artificial density coupling conditions seem to be the physically correct ones. More precisely, we expect to have a shock in the incoming pipeline and rarefaction waves in the outgoing pipelines due to turbulence at the junction.
	\item The most energy is dissipated at the junction when the artificial density coupling condition is imposed. 
	\item If the artificial density coupling condition is imposed, the momentum traces at the junction are smaller compared to the other coupling conditions. The momentum traces for the coupling conditions with equal pressure, equal momentum flux or equal stagnation enthalpy differ less strongly in comparison with each other.
\end{itemize}

\subsection{Level sets associated the coupling conditions}
In this section we consider the geometry of level sets corresponding to different coupling conditions. More precisely, we drop the condition on conservation of mass and compute the sets in which the attained boundary values may lie. For a solution to the generalized Riemann problem, the traces $\{(\bar\rho_k,\bar\rho_k\bar u_k)\,|\,k=1,\dots,d \}$ are contained in one of these level sets. 
More precisely, we consider
\begin{equation}\label{eq:DefLevelSetNumSimSubsonic}
	\{(\rho,\rho u)\in D\,|\, \mathcal H(\rho,u)=\mathcal H(\rho_0,u_0)\},
\end{equation}
where $(\rho_0,\rho_0 u_0)\in D$ is a fixed (subsonic) state and $\mathcal H$ denotes the pressure, momentum flux or stagnation enthalpy. For the artificial density coupling condition, we consider the set
\begin{equation*}
	\mathcal V(\rho_*,0),\quad\text{for suitable }\rho_*\ge 0,
\end{equation*}
which coincides with the definition in \eqref{eq:DefLevelSetNumSimSubsonic} in the subsonic case by taking $\mathcal H=R_*$ (see Definition \ref{def:PStar}).
We set $\gamma=1.4$, $\kappa=1$ and $(\rho_0,\rho_0 u_0)=(1,0)$. This choice leads to the level sets displayed in Figure \ref{fig:LevelSetsCC}.\medskip\\
\begin{figure}
  \centering
  \begin{minipage}{.48\linewidth}
    \centering
    \subcaptionbox{Equal pressure}[5cm]{\includegraphics[trim=5cm 10cm 5cm 9cm,clip=true,width=0.6\linewidth,height=0.6\linewidth]{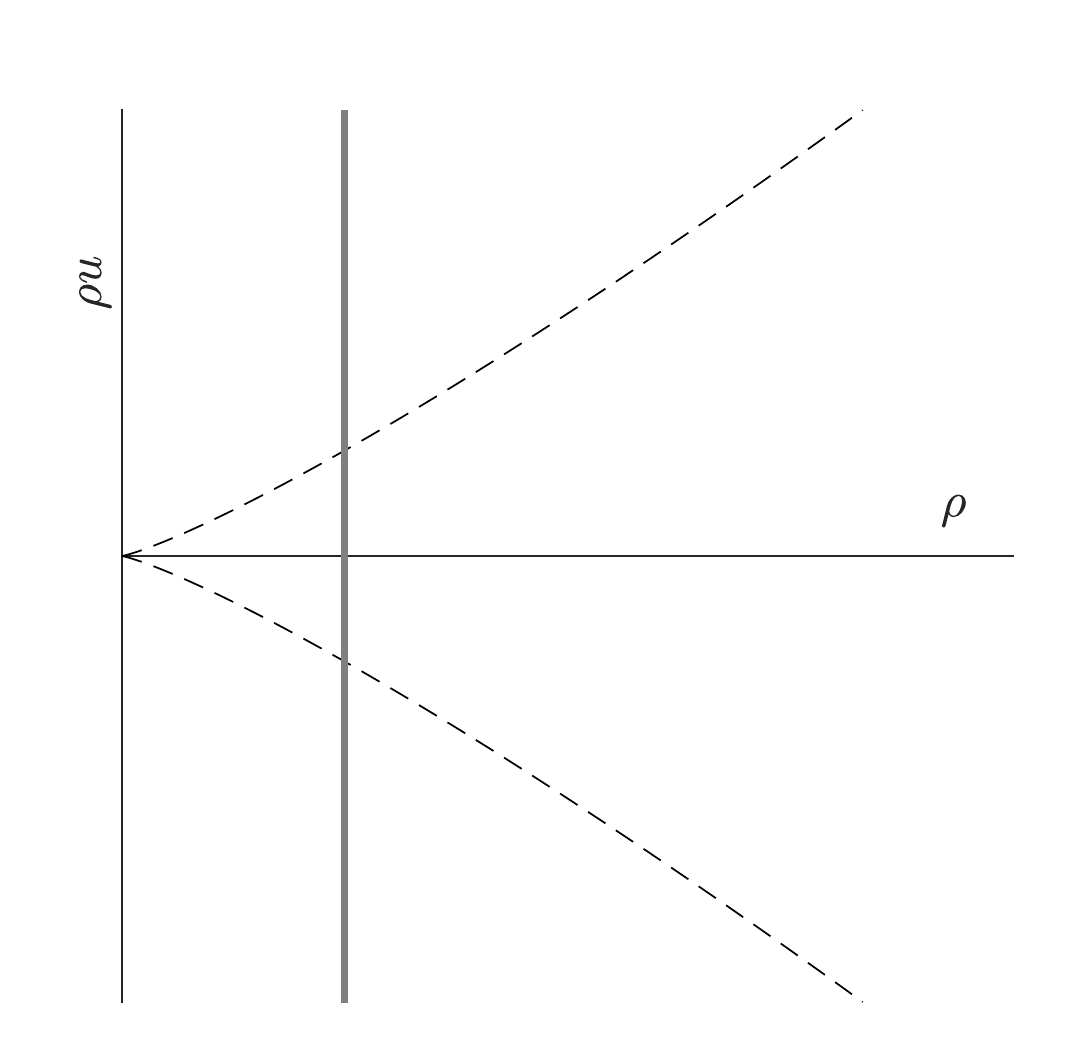}}\\
    \subcaptionbox{Equal momentum flux}[5cm]{\includegraphics[trim=5cm 10cm 5cm 9cm,clip=true,width=0.6\linewidth,height=0.6\linewidth]{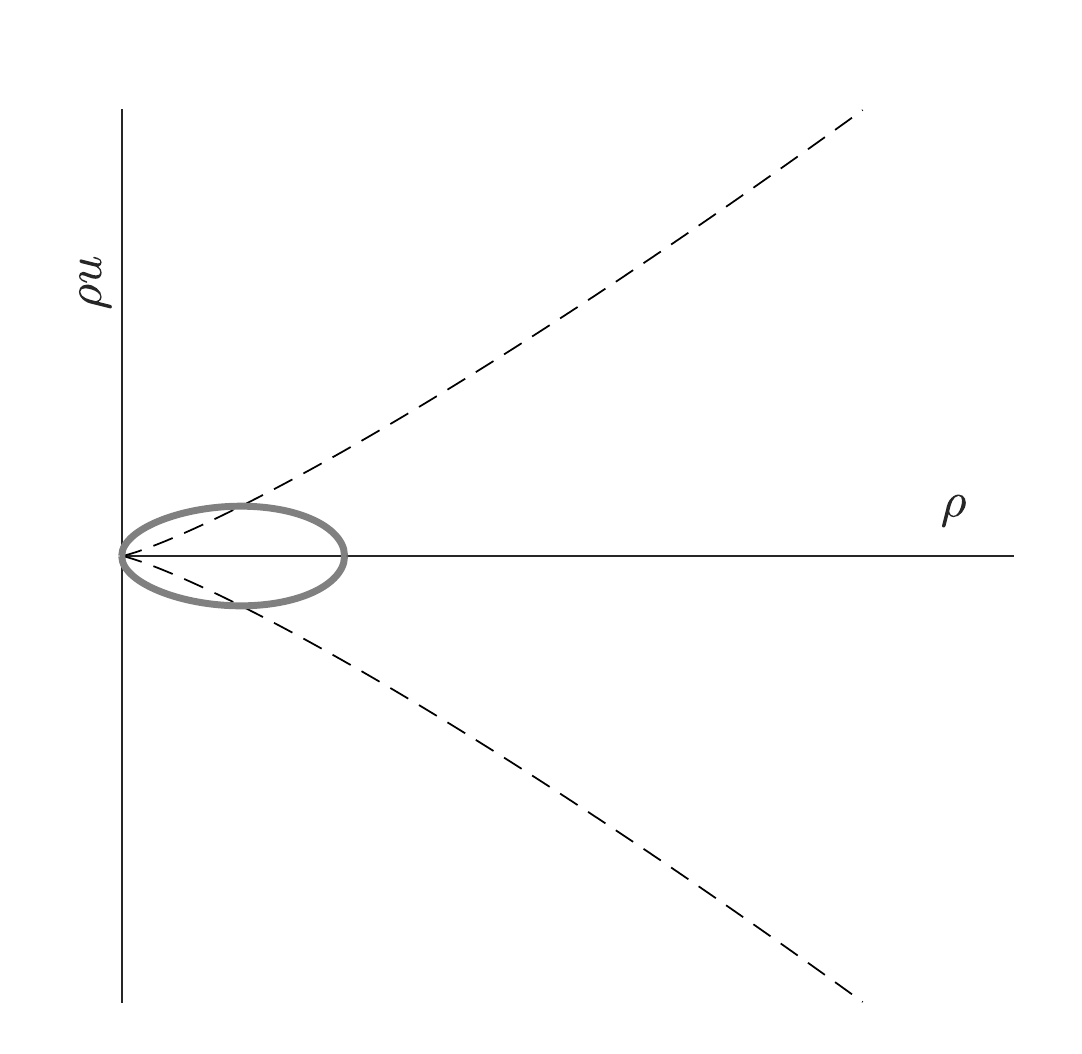}}
  \end{minipage}\quad
  \begin{minipage}{.48\linewidth}
    \centering
    \subcaptionbox{Equal stagnation enthalpy}[5cm]
      {\includegraphics[trim=5cm 10cm 5cm 9cm,clip=true,width=0.6\linewidth,height=0.6\linewidth]{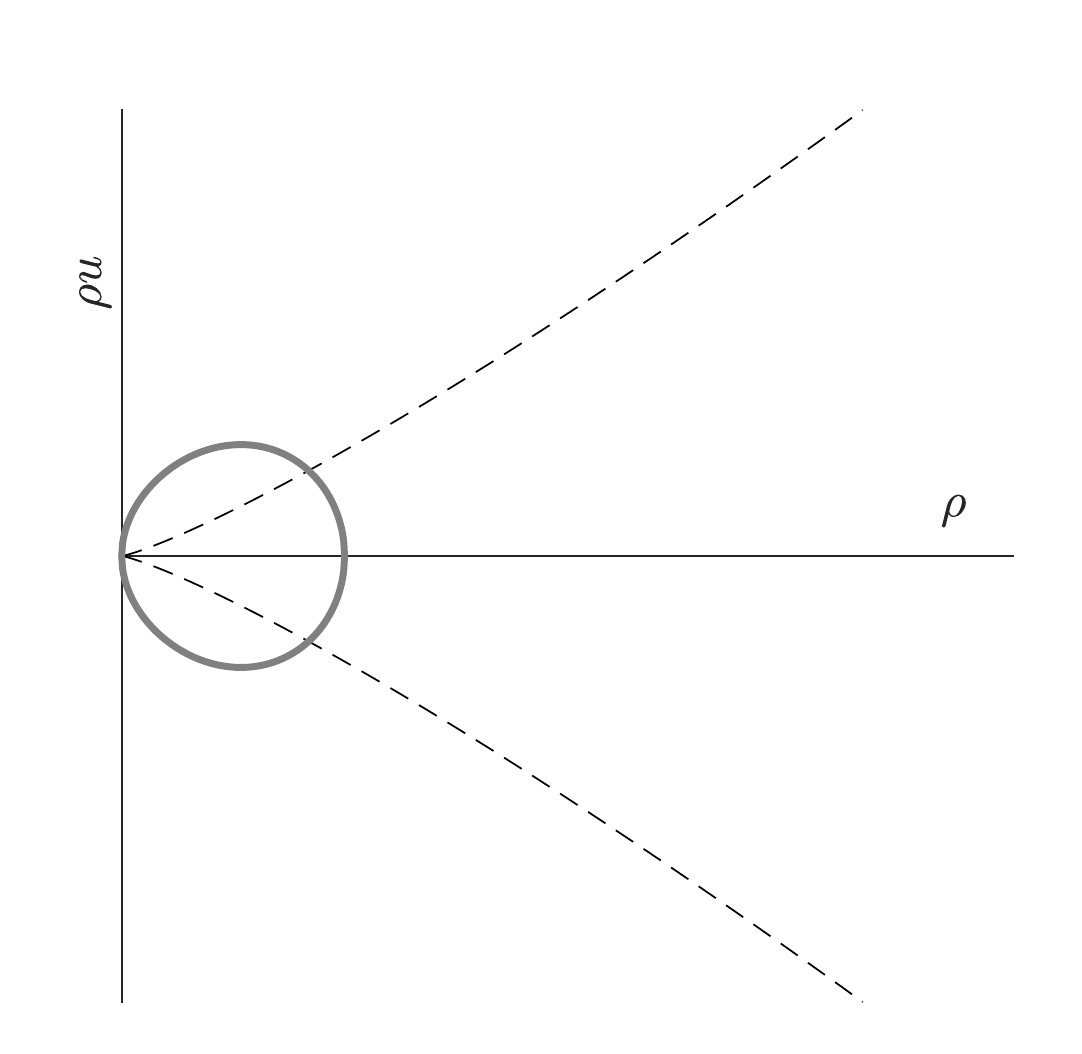}}\\
    \subcaptionbox{Artificial density}[5cm]
      {\includegraphics[trim=5cm 10cm 5cm 9cm,clip=true,width=0.6\linewidth,height=0.6\linewidth]{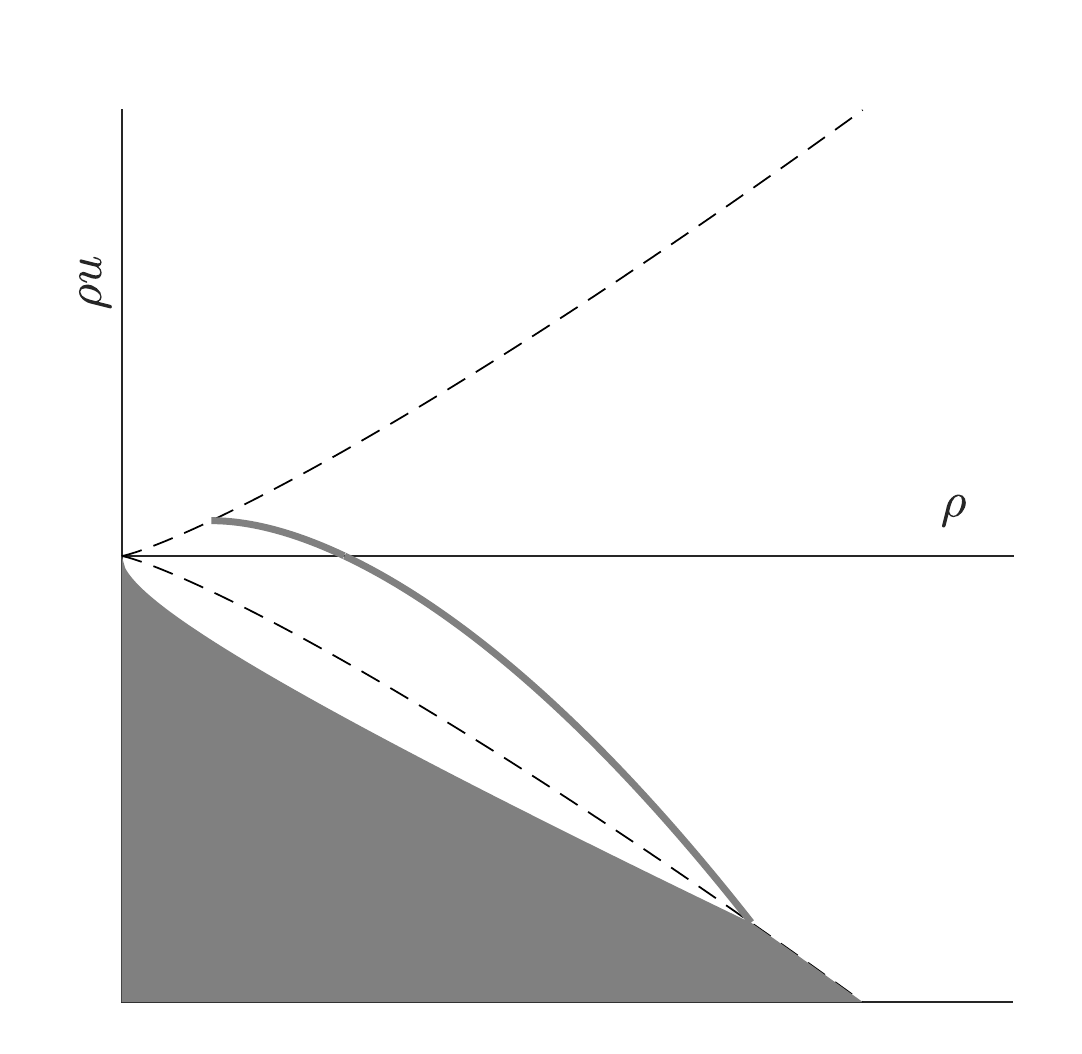}}
  \end{minipage}
  \caption{Level sets for different coupling conditions}\label{fig:LevelSetsCC}
\end{figure}
\noindent
We make the following observations:
\begin{itemize}
	\item The coupling conditions with equal momentum flux and stagnation enthalpy induce bounded level sets.
	\item The level sets induced by equal pressure, momentum flux and stagnation enthalpy are symmetric w.r.t. the $\rho$-axis and have the same tangent at $(1,0)$. This geometric property may lead to asymptotically similar behaviour of the coupling conditions near a stationary solution.
\end{itemize}

\section{Full gas dynamics}
\label{sec:FullGasDynamics}
In this section, we apply the derivation technique to full gas dynamics. We compute the maximum entropy dissipating kinetic coupling condition for kinetic models with the standard Maxwellian, e.g. the Boltzmann equation, the linear Boltzmann equation or the Boltzmann BGK model. We maximize the entropy dissipation and not the energy dissipation since energy is conserved. Again, we define a macroscopic coupling condition which can be formally obtained by a macroscopic limit. Notice that the formal macroscopic limit of the kinetic Boltzmann (type) equations is given by the full compressible Euler equations. The aim of this section is to underline that the presented approach is quite general and can be adopted easily to other hyperbolic systems equipped by a kinetic model and an entropy. Furthermore, the obtained results are very similar to the results for the isentropic gas equations.
\subsection{Derivation of the kinetic coupling condition}
First, we recall some basic definitions and explain the setting. As before, we consider a network of one-dimensional pipelines. We assume that one of the kinetic equations mentioned above is satisfied in the interior of the pipelines. The kinetic Boltzmann (type) equations admit the standard Maxwellian
\begin{equation}
M_{\rho, u, \theta}(\xi)=\frac{\rho}{\sqrt{2 \pi\theta}} \exp\left(\frac{-(u-\xi)^2}{2\theta}\right),
\end{equation}
for $\rho\ge 0$, $u\in \mathbb R$, $\theta > 0$.
Again, we define a kinetic coupling condition at the junction. Since energy is conserved in full gas dynamics, we obtain a second natural condition in addition to the conservation of mass. Nevertheless, two conditions are not enough to select a unique kinetic coupling condition. Therefore, we aim to select the kinetic coupling condition given by
\begin{equation}
\Psi\colon L^1_\mu((-\infty,0)_\xi,[0,\infty))^d\to L^1_\mu((0,\infty)_\xi,[0,\infty))^d;\quad g\mapsto\Psi[g],
\end{equation}
which conserves mass and energy and dissipates as much entropy as possible. More precisely, we minimize
\begin{equation}
\sum_{k=1}^d A_k \int_0^\infty \xi\, \Psi^k[g](\xi) \log \Psi^k[g](\xi)\,\mathrm{d}\xi,
\end{equation}
with respect to
\begin{align}
\sum_{k=1}^d A_k\left(\int_0^\infty \xi\, \Psi^k[g](\xi) \,\mathrm{d}\xi+\int_{-\infty}^0 \xi\, g^k(\xi) \,\mathrm{d}\xi\right)&=0,\label{eq:ConservationMassBoltzmann}\\
\sum_{k=1}^d A_k\left(\int_0^\infty \xi^3\, \Psi^k[g](\xi) \,\mathrm{d}\xi+\int_{-\infty}^0 \xi^3\, g^k(\xi) \,\mathrm{d}\xi\right)&=0.\label{eq:ConservationEngyBoltzmann}
\end{align}
The unique minimizer of this problem is given by 
\begin{equation}
\Psi^k[g](\xi)=M_{\rho_*,0,\theta_*}(\xi),\quad\text{for }\xi>0,
\end{equation}
where $\rho_*\ge 0,\,\theta_*>0$ are chosen such that (\ref{eq:ConservationMassBoltzmann} -- \ref{eq:ConservationEngyBoltzmann}) hold.
The proof works similar to (\ref{eq:HilfMaximumEntropyDissipation}):\\
Since $v\mapsto v \log v$ is convex on $[0,\infty)$ and admits the derivative $v\mapsto \log v+1$, we get
\begin{align}
&\sum_{k=1}^d A_k \int_0^\infty \xi\, \Psi^k[g](\xi) \log \Psi^k[g](\xi)\,\mathrm{d}\xi\nonumber\\
\ge&\sum_{k=1}^d A_k\left( \int_0^\infty \xi\, M_{\rho_*,0,\theta_*}(\xi) \log M_{\rho_*,0,\theta_*}(\xi)\,\mathrm{d}\xi\right)\nonumber\\
&+\sum_{k=1}^d A_k\left(\int_0^\infty \xi\, [\log M_{\rho_*,0,\theta_*}(\xi)+1]\left(\Psi^k[g](\xi)-M_{\rho_*,0,\theta_*}(\xi)\right)\,\mathrm{d}\xi\right)\nonumber\\
=& \sum_{k=1}^d A_k\int_0^\infty \xi\, M_{\rho_*,0,\theta_*}(\xi) \log M_{\rho_*,0,\theta_*}(\xi)\,\mathrm{d}\xi.
\end{align}
The last step follows by
\begin{equation}
\log M_{\rho_*,0,\theta_*}(\xi)+1=\log \frac{\rho}{\sqrt{2 \pi\theta}} +1 -\frac{\xi^2}{2\theta}
\end{equation}
and (\ref{eq:ConservationMassBoltzmann} -- \ref{eq:ConservationEngyBoltzmann}). It can be easily checked that for every $g$ there exists $(\rho_*,0,\theta_*)$ such that (\ref{eq:ConservationMassBoltzmann} -- \ref{eq:ConservationEngyBoltzmann}) hold for $\Psi^k[g](\xi)=M_{\rho_*,0,\theta_*}(\xi)$. By convexity of $v\mapsto v \log v$, it follows that entropy is non-increasing at the junction, i.e.
\begin{equation}
\sum_{k=1}^d A_k\left(\int_0^\infty \xi\, M_{\rho_*,0,\theta_*}(\xi) \log M_{\rho_*,0,\theta_*}(\xi)\,\mathrm{d}\xi+\int_{-\infty}^0 \xi\, g(\xi) \log g(\xi)\,\mathrm{d}\xi\right)\le 0.
\end{equation}

\subsection{The full Euler equations on networks and a new coupling condition}
Next, we consider the macroscopic limit. As mentioned above, the kinetic equation converges formally towards the full Euler equations for ideal polytropic gas given by
\begin{align}
\begin{cases}
\partial_t \rho +\partial_x(\rho u)&=0,\\
\partial_t(\rho u)+\partial_x(\rho u^2+\rho\theta)&=0,\\
\partial_t(\frac{\rho u^2}{2}+\frac{\rho\theta}{2})+\partial_x(\frac{\rho u^3}{2}+\frac{3}{2}\rho u \theta)&=0,
\end{cases}\quad \text{for a.e. }t>0, x\in\mathbb R,
\end{align}
with density $\rho\ge 0$, mean velocity $u\in\mathbb R$, temperature $\theta>0$ and adiabatic index $\gamma=3$. The equations of full gas dynamics model conservation of mass, momentum and energy. As usually, we impose the additional entropy condition
\begin{equation}
\partial_t\left(\rho \log \left( \frac{\rho}{\theta^{1/2}}\right)\right)+\partial_x\left(\rho u \log \left( \frac{\rho}{\theta^{1/2}}\right) \right)\le 0,\quad \text{for a.e. }t>0,x\in\mathbb R.
\end{equation}
The full Euler equations on networks were studied before by several authors \cite{CoMa2010, CoMa2008, He2008, LaMi2018}. We summarize the main ideas of the constructed coupling conditions. Analogous to isentropic gas dynamics, conservation of mass at the junction is imposed
\begin{equation}
	\sum_{k=1}^d A_k\bar{\rho}_k\bar{u}_k=0,\quad\text{a.e. }t>0.
\end{equation}
Since energy is conserved in full gas dynamics, we additionally assume that energy is conserved at the junction, i.e.
\begin{equation}
	\sum_{k=1}^d A_k\overline{\left(\frac{\rho u^3}{2}+\frac{3}{2}\rho u \theta \right)_k}=0.
\end{equation}
There are more conditions needed to single out a unique solution. Most of them are a straight forward extension of a coupling condition for isentropic gas.
We give a short overview of the coupling conditions in the literature:\\
Colombo and Mauri \cite{CoMa2008} introduced equality of momentum flux at the junction
\begin{equation}
\overline{\left(\rho u^2+\rho\theta \right)_k}=\mathcal H_{\mathrm{MF}}(t),\quad\text{for a.e. } t>0, k=1,\dots,d.
\end{equation}
Herty \cite{He2008} used equality of pressure
\begin{equation}
\bar\rho_k\bar\theta_k=\mathcal H_p(t),\quad\text{for a.e. } t>0, k=1,\dots,d.
\end{equation}
Networks consisting of $d=2$ pipelines with different cross-sectional area were studied by Colombo and Marcellini \cite{CoMa2010} with different coupling conditions. One of them is based on a smooth approximation of the discontinuity in the cross-section.
Lang and Mindt \cite{LaMi2018} impose equality of stagnation enthalpy
\begin{equation}
\overline{\left(\frac{u^2}{2}+\frac{3}{2}\theta \right)_k}=\mathcal H_{\mathrm{SE}}(t),\quad\text{for a.e. } t>0, k=1,\dots,d,
\end{equation}
and equality of entropy for traces with outgoing flow
\begin{gather}
\log\left(\frac{\rho}{\theta^{1/2}}\right)=\mathcal H_{\mathrm{S}}(t),\quad\text{for a.e. }t>0, \text{ for }\bar u_k>0,\\
\text{with}\quad \mathcal H_{\mathrm{S}}(t)=\frac{\sum_{\bar u_k<0}A_k \overline{\left(\rho u \log\left(\frac{\rho}{\theta^{1/2}}\right) \right)_k}}{\sum_{\bar u_k<0}A_k \bar \rho_k\bar u_k}.
\end{gather}
These two conditions imply conservation of energy and entropy at the junction. Notice that conservation of entropy at the junction is not consistent with the fact that entropy can be dissipated in full gas dynamics. \medskip\\
In full gas dynamics an additional phenomena appears since the number of ingoing/outgoing characteristics at the junction can change in the subsonic region. This fact makes it more complicated to prove existence and uniqueness results. Nevertheless, we can use the formal arguments in Section \ref{sec:MotivationandDerivation} and the derivation in the previous subsection to define the following new coupling condition for full gas dynamics.
\begin{definition}\label{def:generalizedRiemannFullEuler}
Fix initial data $(\hat\rho_k,\hat u_k,\hat\theta_k)\in D_{3 \times 3}=\{(\rho,u,\theta)\,|\,\rho> 0, u\in\mathbb R,\theta>0 \text{ or }\rho=u=\theta=0\},k=1,\dots,d$. Then, we call $(\rho_k,u_k,\theta_k)\colon (0,\infty)_t\times(0,\infty)_x\to D$ a weak solution to the generalized Riemann problem if the following assertions hold true:
\begin{itemize}
		\item [RP0:] The solution satisfies the initial condition
\begin{equation*}
(\rho_k,u_k,\theta_k)(0+,x)=(\hat{\rho}_k,\hat{u}_k,\hat{\theta}_k)\in D_{3 \times 3},\quad \text{for  }x>0,\,k=1,\dots,d;
\end{equation*}
    \item [RP1:] There exists $(\rho_*,0,\theta_*)\in D_{3\times 3}$ such that $(\rho_k,u_k,\theta_k)$ is equal to the restriction to $x>0$ of the Lax solution to the standard Riemann problem with initial condition
			\begin{equation*}
				(\rho_k,u_k,\theta_k)(0+,x)=
				\begin{cases}
					(\hat{\rho}_k,\hat{u}_k,\hat{\theta}_k),\quad& x>0,\\
					(\rho_*,0,\theta_*),\quad& x<0,
				\end{cases}
			\end{equation*}
			for all $k=1,\dots,d$;
    \item [RP2:] Mass is conserved at the junction
			\begin{equation*}
				\sum_{k=1}^d A_k\bar{\rho}_k\bar{u}_k=0,\quad\text{for all }t>0;
			\end{equation*}
		\item [RP3:] Energy is conserved at the junction
			\begin{equation*}
				\sum_{k=1}^d A_k\overline{\left(\frac{\rho u^3}{2}+\frac{3}{2}\rho u \theta \right)_k}=0,\quad\text{for all }t>0.
			\end{equation*}
\end{itemize}
\end{definition}
Notice, that this condition leads to conservation of mass and energy at the junction by definition. Furthermore, entropy is non-increasing at the junction by the entropy formulation of boundary conditions $\mathcal E(\rho_*,0,\theta_*)$, i.e.
\begin{equation}
\sum_{k=1}^d A_k\overline{\left(\rho u \log\left(\frac{\rho}{\theta^{1/2}}\right) \right)_k}\le 0.
\end{equation}
Therefore, the new coupling condition satisfies some necessary physical properties.

\section{Conclusion}
\label{sec:OutlookRemarks}
We introduced a new coupling condition for isentropic gas and proved existence and uniqueness of solutions to the generalized Riemann and Cauchy problem.\medskip\\
The derivation of the coupling condition is based on the kinetic model and the selection of the unique kinetic coupling condition which conserves mass and dissipates as much energy as possible. The obtained kinetic coupling condition distributes the incoming kinetic data into all pipelines by the same Maxwellian with suitable artificial density and zero speed. Formal arguments lead to a corresponding macroscopic definition to the generalized Riemann problem. In this definition the artificial state with zero speed appears as the (left) initial state for a standard Riemann problem.\medskip\\
In addition to the derivation, we proved physical properties of the coupling condition. The coupling condition ensures that energy is non-increasing at the junction and leads to a maximum principle on the Riemann invariants. Furthermore, a relation of the traces of the stagnation enthalpy at the junction was given. Notice that these properties hold true due to the choice of an artificial state with zero speed.\medskip\\
We gave an example in which the new coupling condition is the only condition producing the physically correct wave types. The solutions to the generalized Riemann problems were computed numerically. Furthermore, we studied level sets related to different coupling conditions and their geometry.\medskip\\
Finally, we considered the coupling condition in view of the model hierarchy of gas dynamics by applying the same approach to full gas dynamics and Boltzmann (type) equations. We took the kinetic coupling conditions with conservation of mass and energy at the junction and maximize the entropy dissipation. This consideration leads to very similar results. In particular, we obtained an artificial state with suitable density and temperature and again with zero speed. \medskip\\
In summary, we defined a new coupling condition, derived several physical and mathematical properties and gave a motivation. Future research may consider more detailed numerical aspects and the rigorous justification of the considerations in Section \ref{sec:MotivationandDerivation}.

\providecommand{\href}[2]{#2}
\providecommand{\arxiv}[1]{\href{http://arxiv.org/abs/#1}{arXiv:#1}}
\providecommand{\url}[1]{\texttt{#1}}
\providecommand{\urlprefix}{URL }

\end{document}